\documentclass[12pt,dvipsnames]{article}

\usepackage{amsmath}
\usepackage{amssymb}
\usepackage{amsthm}
\usepackage[UKenglish]{babel}
\usepackage{bm}
\usepackage[colorlinks=true,allcolors=purple]{hyperref}
\usepackage[capitalize]{cleveref}
	\crefformat{equation}{\textup{#2(#1)#3}}
\usepackage{dsfont}
\usepackage{enumitem}
\usepackage[OT2,OT1]{fontenc}
\usepackage{geometry}
\usepackage[cal=cm,scr=boondoxo]{mathalfa}
\usepackage{pifont}
\usepackage{stmaryrd}
\usepackage{textcomp}
\usepackage[textwidth=3cm,colorinlistoftodos]{todonotes}

\numberwithin{equation}{section}			
\swapnumbers						

\newenvironment{Enumerate}{\begin{enumerate}[label={\rm({\roman*})}]}{\end{enumerate}}

\newcounter{StepsCount}				

\newcommand{\descriptionlabelsave}{}		
\newenvironment{Itemize}{%
	\renewcommand{\descriptionlabelsave}{\descriptionlabel}\renewcommand{\descriptionlabel}{$\triangleright$}%
	\begin{description}[leftmargin=15pt,itemindent=-5.2pt]}{%
	\end{description}\renewcommand{\descriptionlabel}{\descriptionlabelsave}}

\newenvironment{Ilist}{
	\begin{list}{$\triangleright$}{\leftmargin=0pt \labelwidth=11pt \itemindent=\labelwidth%
	\itemsep=5pt\listparindent=\parindent}}{\end{list}}

\theoremstyle{plain}
	\newtheorem{lemma}{Lemma}[section]
	\newtheorem{proposition}[lemma]{Proposition}
	\newtheorem{theorem}[lemma]{Theorem}
	\newtheorem{corollary}[lemma]{Corollary}
\theoremstyle{definition}
	\newtheorem{definition}[lemma]{Definition}
\theoremstyle{remark}
	\newtheorem{remark}[lemma]{Remark}
	\newtheorem{example}[lemma]{Example}
\newenvironment{Lemma}{\begin{lemma}}{\par\noindent\rule{5em}{1pt}\end{lemma}}

\newenvironment{Definition}{\begin{definition}}{\hfill$\blacktriangleleft$\end{definition}}
\newenvironment{Remark}{\begin{remark}}{\hfill$\vartriangleleft$\end{remark}}

\newcommand{\mc}[1]{{\mathcal{#1}}}			
\newcommand{\ms}[1]{{\mathscr{#1}}}			
\newcommand{\mf}[1]{{\mathfrak{#1}}}			
\newcommand{\bb}[1]{{\mathbb{#1}}}			
\newcommand{\ov}{\overline}				
\newcommand{\mr}{\mathring}				

\newcommand{\smmatrix}[4]{\Bigl(			
\begin{smallmatrix}
\hspace*{-0.2ex} #1 \hspace*{0.2ex} & \hspace*{0.2ex} #2 \hspace*{-0.2ex}
\\[0.5ex]
\hspace*{-0.2ex} #3 \hspace*{0.2ex} & \hspace*{0.2ex} #4 \hspace*{-0.2ex}
\end{smallmatrix}
\Bigr)}
\newcommand{\Dummy}{\text{\textvisiblespace\kern1pt}}	
\DeclareMathOperator{\Span}{span}			
\DeclareMathOperator{\Ran}{ran}				
\DeclareMathOperator{\Dom}{dom}				

\newcommand{\DS}{\mid\mkern3mu}				
\newcommand{\DSB}{\mkern4.5mu\Big|\mkern7.5mu}		
\newcommand{\DP}{{\mathop:\kern5pt}}			
\newcommand{\DF}{\colon}				
\newcommand{\DE}{\mathrel{\mathop:}=}			
\newcommand{\DI}{\mathrel{\mathop:}\Leftrightarrow}	
\newcommand{\DD}{\mkern3mu\mathrm{d}}			

\newcommand{\CAS}{&\text{if}\ }				
\newcommand{\CASO}{&\text{otherwise}}			

\renewcommand{\oe}{{\ddot o}}
\newcommand{\Ham}{\bb H}
\newcommand{\HamCC}{\bb H_{\sf cc}}
\newcommand{\HamCP}{\bb H_{\sf cp}}
\newcommand{\HamPC}{\bb H_{\sf pc}}
\newcommand{\HamPP}{\bb H_{\sf pp}}
\newcommand{\Tmax}{T_{\sf max}}
\newcommand{\TDmax}{T_{\Delta,\sf max}}

\usepackage{upgreek}
\usepackage[UKenglish]{isodate}
\newcommand\Gammar{\Gamma_{\textup{\textsf{r}}}}
\newcommand\Gammas{\Gamma_{\textup{\textsf{s}}}}
\newcommand{\mrr}[1]{\overset{\text{\rm\raisebox{-1pt}{\tiny\textmarried}}}{#1}}

\newcommand{\tildeW}{\hspace{0.42ex}\widetilde{\rule[1.5ex]{1.5ex}{0ex}}\hspace{-2.0ex}W}

\begin{document}

\begin{flushleft}
	{\Large\bf The direct spectral problem for indefinite \\[0.5ex] canonical systems}
	\\[5mm]
	\textsc{
	Matthias Langer
	\,\ $\ast$\,\ 
	Harald Woracek 
	}
\end{flushleft}
	{\small
	\textbf{Abstract:}
		For indefinite (Pontryagin space) canonical systems that contain an inner singularity we prove the existence
		of generalised boundary values at the singularity, which are used to formulate interface conditions.
		With the help of such interface conditions we construct the monodromy matrix of the canonical system 
		and write it as a product of matrices, which separates the contributions of the Hamiltonian function 
		and the finitely many discrete parameters that are associated with the singularity.
	}
\begin{flushleft}
	{\small\textbf{AMS MSC 2020:} 34B05, 47B50, 34B20
	}
\end{flushleft}

\begin{quote}
	\textit{%
		This paper is dedicated to the memory of our teacher Heinz Langer, to whom we owe more than we can express in words.
	}
\end{quote}


%
\section{Introduction}
%

A two-dimensional canonical system is a differential equation of the form 
\begin{equation}
\label{D41}
	y'(t)=zJH(t)y(t), \qquad t\in(s_-,s_+),
\end{equation}
where $z$ is a complex parameter, $J$ is the symplectic matrix $J\DE\smmatrix 0{-1}10$, $(s_-,s_+)$ is a finite or infinite 
non-empty interval, and where $H$ is a measurable function taking real, positive semi-definite $2\times 2$-matrices as values. 
The function $H$ is called the Hamiltonian of the system. 

A natural condition for solutions to exist is that $H$ is locally integrable on $(s_-,s_+)$. 
The behaviour of the system towards the endpoints $s_-,s_+$ highly depends on the growth of $H$ towards these endpoints.  
Let us consider the case when $H\in L^1((s_-,s_+),\bb R^{2\times 2})$.  Then every solution has an absolutely
continuous extension to $[s_-,s_+]$ and the initial value problem
\begin{equation}
\label{D42}
	\left\{
	\begin{array}{l}
		\dfrac{\partial}{\partial t}W_H(t,z)J = zW_H(t,z)H(t),\qquad t\in[s_-,s_+],
		\\[2.5ex]
		W_H(s_-,z)=I.
	\end{array}
	\right.
\end{equation}
has a unique solution $W_H\DF[s_-,s_+]\times\bb C\to\bb C^{2\times2}$ (for technical reasons we pass to transposes:
the transposes of the rows of $W_H$ satisfy \eqref{D41}). 
The matrix function $W_H\DE W_H(s_+,\Dummy)$ is called the \emph{monodromy matrix} of $H$ 
and is a central object in the (spectral) theory of canonical systems with integrable $H$.

The monodromy matrix $W_H$ is an entire $2\times 2$-matrix-valued function, is real along the real axis, satisfies
$W_H(0)=I$ and $\det W_H(z)=1$ for $z\in\bb C$, and is $J$-expansive, i.e.\ the kernel 
\begin{equation}
\label{D43}
	K(z,w)\DE\frac{W_H(z)JW_H(w)^*-J}{z-\ov w}, \qquad z,w\in\bb C,
\end{equation}
is positive.  Let us denote the set of all matrix functions with these properties by $\mc M_0$.  
The class $\mc M_0$ is very natural because the following inverse theorem holds: 
given $W\in\mc M_0$, there exist (an essentially unique) $H$ as above such that $W$ is the monodromy matrix of $H$.
Note that the solution is in general non-constructive; only in some cases constructive algorithms are available. 

A connection between a Hamiltonian and its monodromy matrix exists also on an operator-theoretic level.  Equation \eqref{D41}
gives rise to a differential operator in a space $L^2(H)$ of $2$-vector-valued functions, the monodromy matrix gives rise to an
operator in a reproducing kernel Hilbert space of $2$-vector-valued entire functions or (using Krein's $Q$-function theory) a
multiplication operator in an $L^2$-space.  All these models are unitarily equivalent in a natural way. 

Let us now relax the condition that the kernel \eqref{D43} is positive, and proceed to a sign-indefinite setting: we denote by 
$\mc M_{<\infty}$ the set of all entire $2\times 2$-matrix-valued functions $W$ that are real along the real axis, satisfy 
$W(0)=I$, $\det W=1$, and have the property that the kernel \eqref{D43} has a finite number of negative squares. 

On the operator-theoretic side we have analogues of the reproducing kernel model and the $Q$-function model, the only
difference being that the operators act in a Pontryagin space instead of a Hilbert space. Several decades ago M.G.~Krein and
H.~Langer posed the question%
\footnote{The second author learned about this question already during his PhD time from his supervisor H.~Langer.}
whether there also exists an analogue of the differential operator model that resembles a canonical
system, so that a matrix $W\in\mc M_{<\infty}$ can be interpreted as the monodromy matrix of a `kind of differential operator'
given by a `kind of indefinite Hamiltonian'.  An affirmative answer was given in a series of papers 
\cite{kaltenbaeck.woracek:db}--\cite{kaltenbaeck.woracek:p6db}.
The analogue of a Hamiltonian $H\in L^1((s_-,s_+),\bb R^{2\times 2})$ consists of a measurable function $H$ that takes
real, positive semi-definite $2\times 2$-matrices as values and is locally integrable on a set 
$[s_-,s_+]\setminus\{\sigma_1,\ldots,\sigma_m\}$ where $s_-<\sigma_1<\ldots<\sigma_m<s_+$, and of a certain finite
number of parameters attached to each of the exceptional points $\sigma_i$. 

The intuition behind the construction from \cite{kaltenbaeck.woracek:db}--\cite{kaltenbaeck.woracek:p6db} is that one solves a
canonical system on each connected component of $[s_-,s_+]\setminus\{\sigma_1,\ldots,\sigma_m\}$ and fits the separate solutions
together in a certain way that takes into account the asymptotic behaviour of $H$ towards the singularities $\sigma_i$ and
the additional data attached to them. Phrased slightly differently, the monodromy matrix $W_H$ can be obtained by
solving the initial value problem \eqref{D42} on the interval $[s_-,\sigma_1)$, then jump over the singularity $\sigma_1$ by which a
certain twist may happen, then solve the equation on $(\sigma_1,\sigma_2)$, jump over $\sigma_2$, and so on. However, the 
construction of the operator model does not at all expose this intuition; it is purely operator-theoretic relying on the tools from 
\cite{jonas.langer.textorius:1992,krein.langer:1978}.

The main goal of this paper is to bridge that gap.  We show that `jumping over singularities' can be realised by an
interface condition that connects both sides of the singularities.  The basis for this result is laid out in \cref{D115},
where generalised boundary values at a singularity are constructed.
The actual construction of the `continuation' of a solution is contained in \cref{D123}. 

We close this introduction with two important remarks.
\begin{Ilist}
\item 
	We restrict all considerations in the paper to the essential building blocks of indefinite Hamiltonians, so-called
	`elementary indefinite Hamiltonians of kind (A)'.  
	These cover most cases of indefinite Hamiltonians with one singularity.  The remaining cases with one singularity,
	namely `elementary indefinite Hamiltonians of types (B) and (C)', are very simple, and the forward problem 
	is solved explicitly in \cite[Proposition~4.31]{kaltenbaeck.woracek:p5db}.
	For indefinite Hamiltonians with more than one singularity one can paste several elementary indefinite Hamiltonians; see 
	\cite[\S3\,c,d,e]{kaltenbaeck.woracek:p5db}. 
\item
	An essential step in order to reach our goals is to give a unitarily equivalent form of the operator model as an
	(explicit) finite-dimensional perturbation of the natural differential operator induced by \eqref{D41}. This is done in
	\cref{D38} and is an extension of a corresponding result given in \cite{langer.woracek:esmod}.
\end{Ilist}

%
\section{Reminder about canonical systems I. \\ The positive definite case}
\label{D150}
%

In this section we recall definitions and facts about canonical systems with a positive semi-definite Hamiltonian. 
Besides the standard notions (for which we refer to, e.g.\ \cite{hassi.snoo.winkler:2000,romanov:1408.6022v1,remling:2018}), 
we specifically deal with some notions that are needed to introduce the Pontryagin space analogue of 
such systems (these are collected from \cite{kaltenbaeck.woracek:p4db}).

\subsection{Canonical systems with positive Hamiltonian}
\label{D151}

The equation we study is given as follows.

\begin{Definition}
\label{D1}
	Let $s_-,s_+\in\bb R\cup\{\infty\}$ with $-\infty<s_-<s_+\leq\infty$, and let $H\DF(s_-,s_+)\to\bb R^{2\times 2}$. 
	We call $H$ a \emph{Hamiltonian on $(s_-,s_+)$} if 
	\begin{Enumerate}
	\item $H(t)\ge 0$ for a.a.\ $t\in(s_-,s_+)$,
	\item $H$ is measurable and locally integrable on $(s_-,s_+)$,
	\item $\{t\in(s_-,s_+)\DS H(t)=0\}$ has measure $0$.
	\end{Enumerate}
	We denote the set of all Hamiltonians on $(s_-,s_+)$ by $\Ham(s_-,s_+)$.

	Let $H\in\Ham(s_-,s_+)$. The \emph{canonical system} with Hamiltonian $H$ is the equation 
	\begin{equation}
	\label{D2}
		y'(t)=zJH(t)y(t)
	\end{equation}
	where 
	\[
		J\DE\begin{pmatrix} 0 & -1\\ 1 & 0\end{pmatrix}\quad\text{and}\quad z\in\bb C.
	\]
	A \emph{solution} $y(t)$ of the canonical system \eqref{D2} is a locally absolutely continuous function 
	$y\DF(s_-,s_+)\to\bb C^2$ that satisfies \eqref{D2} for a.a.\ $t\in(s_-,s_+)$. 

	When the entries of $H$ are needed, we write 
	\begin{equation}
	\label{D44}
		H(t)=\begin{pmatrix} h_1(t) & h_3(t) \\ h_3(t) & h_2(t) \end{pmatrix}.
	\end{equation}
\end{Definition}

\noindent
Intervals where a Hamiltonian is of a particularly simple form play a---sometimes exceptional---role. 
For $\phi\in\bb R$ we denote by $\xi_\phi$ the vector 
\[
	\xi_\phi\DE\binom{\cos\phi}{\sin\phi}.
\]

\begin{Definition}
\label{D3}
	Let $H$ be a Hamiltonian on $(s_-,s_+)$, and let $a,b\in\bb R$ with $s_-\le a<b\le s_+$. 
	Then the interval $(a,b)$ is called \emph{$H$-indivisible} if there exists $\phi\in\bb R$ such that 
	$\Ran H(t)=\Span\{\xi_\phi\}$ for a.a.\ $t\in(a,b)$. 

	If $(a,b)$ is $H$-indivisible, the value of $\phi$ above is determined up to integer multiples of $\pi$, and we call it
	(strictly, the equivalence class modulo $\pi$) the \emph{type} of the indivisible interval $(a,b)$. 
\end{Definition}

\noindent
On every compact subset of its domain $(s_-,s_+)$ a Hamiltonian $H$ is, by definition, integrable,
and this guarantees existence and uniqueness of solutions of the canonical system with Hamiltonian $H$. 
At the boundary points $s_\pm$ the Hamiltonian may or may not be integrable, and the behaviour of solutions depends on this. 

\begin{Definition}
\label{D4}
	Let $H\in\Ham(s_-,s_+)$. We say that $H$ is in \emph{limit circle case at $s_-$} if $H$ is integrable on some (and
	hence every) interval $(s_-,a)$ where $a\in(s_-,s_+)$.  If $H$ is not in limit circle case at $s_-$, we say that $H$ is
	in \emph{limit point case at $s_-$}. The completely analogous terminology applies at the boundary point $s_+$. 

	We include these case distinctions into our notation by partitioning $\Ham(s_-,s_+)$ as the disjoint union 
	of the four sets
	\[
		\HamCC(s_-,s_+),\quad \HamCP(s_-,s_+),\quad \HamPC(s_-,s_+),\quad \HamPP(s_-,s_+),
	\]
	where $\HamCC(s_-,s_+)$ is the set of all Hamiltonians that are in limit circle case at $s_-$ and $s_+$, 
	$\HamCP(s_-,s_+)$ is the set of all Hamiltonians that are in limit circle case at $s_-$ and in limit point case at $s_+$, 
	etc.
\end{Definition}

\noindent
If $H$ is in limit circle case at $s_-$, then every solution of the canonical system has an extension to $[s_-,s_+)$ that is
absolutely continuous on every interval $[s_-,a)$ where $a\in(s_-,s_+)$. Furthermore, the initial value problem prescribing an 
initial value at $s_-$ is uniquely solvable for every initial value. The analogous statements hold for $s_+$ in place of $s_-$. 

In the next definition we pass to transposes compared to \eqref{D2}; this has certain technical advantages.

\begin{Definition}
\label{D5}
	Let $H\in\HamCC(s_-,s_+)\cup\HamCP(s_-,s_+)$. Then we denote by $W_H\DF[s_-,s_+)\times\bb C\to\bb C^{2\times 2}$ the
	unique solution of
	\begin{align}
		& \frac{\partial}{\partial t}W_H(t,z)J=zW_H(t,z)H(t)\qquad\text{for $t\in(s_-,s_+)$ a.e.},
		\label{D70}
		\\[1ex]
		& W_H(s_-,z)=I,
		\label{D71}
	\end{align}
	and refer to $W_H$ as the \emph{fundamental solution of $H$}. 

	If $H\in\HamCC(s_-,s_+)$, then $W_H(\Dummy,z)$ exists up to $s_+$. We write 
	\[
		W_H(z)\DE W_H(s_+,z)
	\]
	and speak of $W_H$ as the \emph{monodromy matrix of $H$}.
\end{Definition}

\subsection{The operator model}
\label{D152}

The canonical system \eqref{D2} is the formal eigenvalue equation of a certain operator (in general, a linear relation) in a 
certain Hilbert space. This space is essentially the usual $L^2$-space with respect to the matrix measure $H(t)\DD t$, but also
takes care of indivisible intervals in an appropriate way. 

\begin{Definition}
\label{D6}
	Let $H\in\Ham(s_-,s_+)$. 
	\begin{Enumerate}
	\item 
		We define an equivalence relation $=_H$ on the set of all measurable functions of $(s_-,s_+)$ into $\bb C^2$ as 
		\[
			f_1=_Hf_2 \;\DI\; Hf_1=Hf_2 \; \text{ a.e.}
		\]
		We denote by $f/_{=_H}$ the equivalence class of $f$.
	\item 
		The space $L^2(H)$ is the set of all equivalence classes modulo $=_H$ of measurable functions 
		$f\DF(s_-,s_+)\to\bb C^2$ that satisfy 
		\begin{Itemize}
		\item 
			${\displaystyle
			\int_{s_-}^{s_+}f^*(t)H(t)f(t)\DD t<\infty
			}$;
		\item 
			for every $H$-indivisible interval $(a,b)$ the function $\xi_\phi^*f$ is constant a.e.\ on $(a,b)$ where
			$\phi$ is the type of $(a,b)$. 
		\end{Itemize}
	\item 
		An inner product $(\Dummy,\Dummy)_H$ on $L^2(H)$ is defined by
		\[
			(f,g)_H \DE \int_{s_-}^{s_+}g^*(t)H(t)f(t)\DD t.
		\]
	\end{Enumerate}
\end{Definition}

\noindent
The \emph{maximal operator} (or relation) associated with a canonical system is defined by means of its graph as follows.

\begin{Definition}
\label{D7}
	Let $H\in\Ham(s_-,s_+)$. Then we define 
	\begin{align*}
		\Tmax(H) &\DE \Big\{(f,g)\in L^2(H) \times L^2(H)\DSB 
		\\
		&\hspace*{12ex} \exists\hat f\text{ locally a.c.}\DP \hat f/_{=_H}=f \;\wedge\; \hat f'=JHg\text{ a.e.}\Big\}.
	\end{align*}
	Note here that we may write $Hg$ since this expression is independent of the choice of the representative of the equivalence
	class $g$.
\end{Definition}

\noindent
If $H$ is in limit circle case at $s_-$, then every locally absolutely continuous representative $\hat f$ 
as in the definition of $\Tmax(H)$ has an extension to $[s_-,s_+)$, which is absolutely continuous on every interval $[s_-,a)$ 
where $a\in(s_-,s_+)$.  
The analogous statement holds with $s_-$ replaced by $s_+$.  This justifies the definition of boundary values.

\begin{Definition}
\label{D8}
	\phantom{}
	\begin{Enumerate}
	\item Let $H\in\HamCC(s_-,s_+)\cup\HamCP(s_-,s_+)$. Then we define 
		\begin{align*}
			\Gamma_-(H) &\DE \Big\{\big((f,g),c\big)\in (L^2(H)\times L^2(H))\times\bb C^2\DSB 
			\\
			&\hspace*{5ex} \exists\hat f\text{ locally a.c.}\DP 
			\hat f/_{=_H}=f \;\wedge\; \hat f'=JHg\text{ a.e.} \;\wedge\; \hat f(s_-)=c\Big\}.
		\end{align*}
	\item Let $H\in\HamCC(s_-,s_+)\cup\HamPC(s_-,s_+)$. Then we define 
		\begin{align*}
			\Gamma_+(H) &\DE \Big\{\big((f,g),c\big)\in (L^2(H)\times L^2(H))\times\bb C^2\DSB 
			\\
			&\hspace*{5ex} \exists\hat f\text{ locally a.c.}\DP 
			\hat f/_{=_H}=f \;\wedge\; \hat f'=JHg\text{ a.e.} \;\wedge\; \hat f(s_+)=c\Big\}.
		\end{align*}
	\end{Enumerate}
\end{Definition}

\begin{Remark}\label{D72}
	Provided that the interval $(s_-,s_+)$ is not $H$-indivisible, for any element $(f,g)\in\Tmax(H)$ there
	exists a unique representative $\hat f$ as in the definition of $\Tmax(H)$; see \cite[Lemma~3.5]{hassi.snoo.winkler:2000}.
\end{Remark}

\begin{Remark}\label{D76}
	Sometimes we need Green's identity in the following form: if $f$ and $u$ are
	absolutely continuous functions on $[x_1,x_2]$ where $s_-\le x_1<x_2\le s_+$ and $g,\,v$
	are such that
	\[
		f' = JHg, \qquad u' = JHv, \qquad \text{a.e.\ on}\;\; (x_1,x_2),
	\]
	then
	\begin{equation}
	\label{D108}
		\int_{x_1}^{x_2}u^*Hg - \int_{x_1}^{x_2}v^*Hf
		= u(x_1)^*Jf(x_1) - u(x_2)^*Jf(x_2);
	\end{equation}
	see \cite[Remark~2.20]{kaltenbaeck.woracek:p4db}.
\end{Remark}

\subsection{The conditions {\sf(I)} and {\sf(HS)}}
\label{D153}

We introduce two conditions on a Hamiltonian $H$ which are needed to prepare the ground for the indefinite setting. 

\begin{Definition}
\label{D9}
	Let $H\in\HamCP(s_-,s_+)$. 
	\begin{Enumerate}
	\item We say that $H$ satisfies the condition \textsf{(I)} if 
		\[
			\int_{s_-}^{s_+}\binom 10^*H(t)\binom 10\DD t<\infty.
		\]
	\item We say that $H$ satisfies the condition {\sf(HS)} if there exists $\phi\in\bb R$ such that 
		\begin{align}
			& \int_{s_-}^{s_+}\xi_\phi^*H(t)\xi_\phi\DD t<\infty,
			\nonumber
			\\[2mm]
			& \int_{s_-}^{s_+}\xi_{\phi+\frac\pi 2}^*\Big(\int_{s_-}^tH(s)\DD s\Big)\xi_{\phi+\frac\pi 2}\cdot
			\xi_\phi^*H(t)\xi_\phi\DD t<\infty.
			\label{D45}
		\end{align}
	\end{Enumerate}
	For $H\in\HamPC(s_-,s_+)$ the conditions {\sf(I)} and {\sf(HS)} are defined analogously, the only difference being that
	in \cref{D45} the expression within the round brackets is replaced by 
	\[
		\Big(\int_t^{s_+}H(s)\DD s\Big).
	\]
\end{Definition}

\noindent
We note that $H\in\HamCP(s_-,s_+)\cup\HamPC(s_-,s_+)$ satisfies {\sf(I)} and {\sf(HS)} if and only if {\sf(HS)} 
holds with $\phi=0$.  
This is true since $H\notin\HamCC(s_-,s_+)$, and hence there exists at most one $\phi\in[0,\pi)$ with 
$\int_{s_-}^{s_+}\xi_\phi^*H(t)\xi_\phi\DD t<\infty$. 
With the notation from \cref{D44} the condition {\sf(I)}\,$\wedge$\,{\sf(HS)} reads as
\[
	\int_{s_-}^{s_+} h_1(t)\DD t < \infty, \qquad \int_{s_-}^{s_+}\Bigl(\int_{s_-}^t h_2(s)\DD s\Bigr)h_1(t)\DD t < \infty
\]
for $H\in\HamCP(s_-,s_+)$ and analogously for $H\in\HamPC(s_-,s_+)$.

The condition {\sf(HS)} has a very clear operator theoretic meaning: it says that the resolvents of self-adjoint restrictions of 
$\Tmax(H)$ are of Hilbert--Schmidt class. 

\subsection{{\boldmath$H$}-polynomials}
\label{D154}

Assume that we have a Hamiltonian $H\in\HamCP(s_-,s_+)$. Then we define an integral operator $\mc I$ acting on the set of all
measurable functions $f\DF[s_-,s_+)\to\bb C^2$ such that $Hf$ is integrable on every interval $[s_-,a)$ where $a\in(s_-,s_+)$. 
We set
\[
	(\mc If)(t)\DE\int_{s_-}^t JH(s)f(s)\DD s\qquad\text{for }t\in[s_-,s_+).
\]
Note that the function $\mc If$ is always continuous. In particular, for each $f$ in the domain of $\mc I$, all iterates 
$\mc I^nf$ are well-defined. 

\begin{Definition}
\label{D10}
	Let $H\in\HamCP(s_-,s_+)$. We call every function of the form 
	\begin{equation}
	\label{D11}
		f=\sum_{k=0}^n\mc I^k\binom{\alpha_k}{\beta_k},
	\end{equation}
	where $n\in\bb N_0$ and $\alpha_k,\beta_k\in\bb C$, an \emph{$H$-polynomial}.

	Moreover, we denote the set of all $H$-polynomials by $\bb C^2[\mc I]$. 
\end{Definition}

\noindent
Note that, in general, neither the number $n$ nor the coefficients $\binom{\alpha_k}{\beta_k}$ 
in the representation \eqref{D11} are uniquely determined by the function $f$. 

\begin{Remark}
\label{D12}
	Caution!

	In \cref{D10} we deviate from the terminology used in \cite{kaltenbaeck.woracek:p4db}:
	there only those functions of the form \eqref{D11} that belong to $L^2(H)$ were called $H$-polynomials. 

	We believe that the above \cref{D10} is more natural.  In fact, it perfectly fits the analogy with usual polynomials:
	every polynomial $p\in\bb C[t]$ is nothing but a sum of iterates of the usual Volterra operator $\int_0^t f(s)\DD s$ 
	applied to some constants.
\end{Remark}

\noindent
Under the assumption that $H$ satisfies {\sf(I)} and {\sf(HS)} already `half of the' functions \eqref{D11} belong to $L^2(H)$. 
The following fact is shown in \cite[Corollary~3.5]{kaltenbaeck.woracek:p4db}.

\begin{Lemma}
\label{D13}
	Let $H\in\HamCP(s_-,s_+)$ and assume that $H$ satisfies {\sf(I)} and {\sf(HS)}. 
	Then there exists a unique sequence $(\rho_n)_{n=0}^\infty$ of numbers $\rho_n\in\bb C$ such that $\rho_0=0$ and 
	\[
		\forall n\geq 1\DP \mc I^n\binom 10+\sum_{k=0}^{n-1}\mc I^k\binom 0{\rho_{n-k}}\in L^2(H).
	\]
\end{Lemma}

\noindent
$H$-polynomials whose `leading coefficient' is $\binom 01$ and that belong to $L^2(H)$ may or may not exist.

\begin{Definition}
\label{D14}
	Let $H\in\HamCP(s_-,s_+)$ and assume that $H$ satisfies {\sf(I)} and {\sf(HS)}.  Then we set 
	\begin{align*}
		\Delta(H) &\DE \inf\Big\{n\in\bb N_0\DSB 
		\\
		&\hspace*{8ex} \exists\binom{\alpha_1}{\beta_1},\ldots,\binom{\alpha_n}{\beta_n}\in\bb C^2\DP
		\mc I^n\binom 01+\sum_{k=0}^{n-1}\mc I^k\binom{\alpha_{n-k}}{\beta_{n-k}}\in L^2(H)\Big\}.
	\end{align*}
	Here the infimum of the empty set is understood as $\infty$.
\end{Definition}

\noindent
Since $H\notin\HamCC(s_-,s_+)$, we always have $\Delta(H)\ge 1$. 

The following fact is shown in \cite[Section~3]{kaltenbaeck.woracek:p4db}.

\begin{Lemma}
\label{D15}
	Let $H\in\HamCP(s_-,s_+)$.  Assume that $H$ satisfies {\sf(I)} and {\sf(HS)}, and that $\Delta(H)<\infty$. 
	Then there exists a unique sequence $(\omega_n)_{n=0}^\infty$ of numbers $\omega_n\in\bb C$, such that $\omega_0=1$ and 
	\begin{equation}
	\label{D16}
		\forall n\geq\Delta(H)\DP \mc I^n\binom 01+\sum_{k=0}^{n-1}\mc I^k\binom 0{\omega_{n-k}}\in L^2(H).
	\end{equation}
\end{Lemma}

\noindent
Note that, due to our definition that $\mf w_0=\binom 01$, we can write
\[
	\mc I^n\binom 01+\sum_{k=0}^{n-1}\mc I^k\binom 0{\omega_{n-k}}=\sum_{k=0}^n\mc I^k\binom 0{\omega_{n-k}}.
\]

\begin{Definition}
\label{D17}
	Let $H\in\HamCP(s_-,s_+)$. Assume that $H$ satisfies {\sf(I)} and {\sf(HS)}, and that $\Delta(H)<\infty$. 
	Let $(\omega_n)_{n=0}^\infty$ be the unique sequence from \cref{D15}.  Then, for every $n\in\bb N_0$, we set
	\begin{equation}\label{D62}
		\mf w_n\DE \mc I^n\binom 01+\sum_{k=0}^{n-1}\mc I^k\binom 0{\omega_{n-k}}.
	\end{equation}
	Moreover, we set $\mf w_{-1}\DE0$.
\end{Definition}

\begin{Remark}
\label{D63}
\rule{0ex}{1ex}
\begin{Enumerate}
\item
	It follows from \eqref{D16} and the definition of $\mc I$ that
	\begin{equation}
	\label{D64}
		\mf w_n'=JH\mf w_{n-1}, \quad \mf w_n(s_-)=\binom{0}{\omega_n}, \qquad n\in\bb N_0.
	\end{equation}
\item
	When $H$ is diagonal on $(s_-,s_+)$, the functions $\mf w_n$ can be determined more explicitly;
	see \cite[Theorem~3.7]{winkler.woracek:del}\footnote{%
		In \cite[(3.8)]{winkler.woracek:del} a minus sign is missing in the formula for $\mf w_{2n+1}$.}.
	Define scalar functions ${\sf w}_n$, $n\in\bb N_0$ by
	\begin{align}
		\sf w_0(t) &= 1,
		\label{D147}
		\\[1ex]
		{\sf w}_{n+1}(t) &= 
		\begin{cases}
			-\int_{s_-}^t h_2(s){\sf w}_n(s)\DD s & \text{if $n$ is even},
			\\[2ex]
			-\int_t^{s_+} h_1(s){\sf w}_n(s)\DD s & \text{if $n$ is odd};
		\end{cases}
		\label{D148}
	\end{align}
	then
	\begin{equation}\label{D149}
		\mf w_n(t) = 
		\begin{cases}
			\binom{0}{{\sf w}_n(t)} & \text{if $n$ is even},
			\\[1ex]
			\binom{{\sf w}_n(t)}{0} & \text{if $n$ is odd}.
		\end{cases}
	\end{equation}
\end{Enumerate}
\end{Remark}

\begin{Remark}
\label{D18}
	If $H\in\HamPC(s_-,s_+)$, we can do the same as elaborated above, simply by exchanging the
	roles of left and right endpoints. Technically this can be done either 
	by applying the above to the Hamiltonian $(t\mapsto H(-t))\in\HamCP(-s_+,-s_-)$, or by considering the integral operator 
	\[
		(\tilde{\mc I}f)(t)\DE-\int_t^{s_+} JH(s)f(s)\DD s\qquad\text{for }t\in(s_-,s_+].
	\]
	and defining the number $\Delta(H)$ and the sequences $(\rho_n)_{n=0}^\infty$ and $(\omega_n)_{n=0}^\infty$
	in the analogous way.

	Let us make this explicit.  Let $H\in\HamPC(s_-,s_+)$ and assume that $H$ satisfies {\sf(I)} and {\sf(HS)}.
	There exists a unique sequence $(\rho_n)_{n=0}^\infty$ of numbers $\rho_n\in\bb C$, such that $\rho_0=0$ and 
	\[
		\forall n\geq 1\DP \tilde{\mc I}^n\binom 10+\sum_{k=0}^{n-1}\tilde{\mc I}^k\binom 0{\rho_{n-k}}\in L^2(H).
	\]
	We set
	\begin{align*}
		\Delta(H) &\DE \inf\Big\{n\in\bb N_0\DSB 
		\\
		&\hspace*{8ex} \exists\binom{\alpha_1}{\beta_1},\ldots,\binom{\alpha_n}{\beta_n}\in\bb C^2\DP
		\tilde{\mc I}^n\binom 01+\sum_{k=0}^{n-1}\tilde{\mc I}^k\binom{\alpha_{n-k}}{\beta_{n-k}}\in L^2(H)\Big\}.
	\end{align*}
	Assume that $\Delta(H)<\infty$. 
	Then there exists a unique sequence $(\omega_n)_{n=0}^\infty$ of numbers $\omega_n\in\bb C$ such that $\omega_0=1$ and 
	\[
		\forall n\geq\Delta(H)\DP \tilde{\mc I}^n\binom 01+\sum_{k=0}^{n-1}\tilde{\mc I}^k\binom 0{\omega_{n-k}}
		\in L^2(H),
	\]
	and we set 
	\[
		\mf w_n\DE \tilde{\mc I}^n\binom 01+\sum_{k=0}^{n-1}\tilde{\mc I}^k\binom 0{\omega_{n-k}}.
	\]
	Further, the relations in \eqref{D147}--\eqref{D149} still hold if $s_-$ is replaced by $s_+$.
\end{Remark}

%
\section{Reminder about canonical systems II. \\ The indefinite case}
\label{D155}
%

In this section we recall definitions and facts about sign-indefinite canonical systems from 
\cite{kaltenbaeck.woracek:p4db}, and give an alternative form of the operator model of an elementary indefinite Hamiltonian of
kind (A) which is similar to \cite[Theorem~2.15]{langer.woracek:esmod}.

\subsection{Elementary indefinite Hamiltonians of kind (A)}
\label{D156}

We introduce the objects that are the essential building blocks of indefinite Hamiltonians and capture `non-trivial'
singularities.

\begin{Definition}
\label{D31}
	An \emph{elementary indefinite Hamiltonian of kind \textup{(A)}}, $\mf h$, is a tuple consisting of data (i)--(iii):
	\begin{Enumerate}
	\item 
		Two finite and non-empty intervals $(s_-,\sigma)$ and $(\sigma,s_+)$, and a Hamiltonian on each of them:
		\[
			H_-\in\HamCP(s_-,\sigma),\quad H_+\in\HamPC(\sigma,s_+).
		\]
		These Hamiltonians are assumed to have the following properties:
		\begin{Itemize}
		\item 
			$H_-$ and $H_+$ satisfy {\sf(I)} and {\sf(HS)}, and 
			\[
				\Delta(H_-)<\infty,\quad \Delta(H_+)<\infty;
			\]
		\item
			for every $s\in[s_-,\sigma)$ the interval $(s,\sigma)$ is not $H_-$-indivisible, 
			or, for every $s\in(\sigma,s_+)$ the interval $(\sigma,s)$ is not $H_+$-indivisible.
		\end{Itemize}
		We write $H\DE(H_-,H_+)$ and $\Delta(H)\DE\max\{\Delta(H_-),\Delta(H_+)\}$.
	\item 
		Numbers $d_0,\ldots,d_{2\Delta(H)-1}\in\bb R$.
	\item 
		A number $\oe\in\bb N_0$ and, if $\oe>0$, numbers $b_1,\ldots,b_\oe\in\bb R$ with $b_1\neq 0$.
	\end{Enumerate}
	If we are given data as above, we write $\mf h=\langle H;\oe,b_j;d_j\rangle$. 
\end{Definition}

\noindent
Let us provide an intuition for this definition. The `function' $H$ is a `cc-Hamiltonian' on $(s_-,s_+)$ that has an inner
singularity at $\sigma$ (limit point at $\sigma$). The `growth' of $H$ towards this singularity is not too fast (limited by
{\sf(HS)} and $\Delta(H)<\infty$), and the singular behaviour of $H$ appears
only one-dimensionally (the direction $\binom 10$ remains integrable by 
{\sf(I)}). The numbers $\oe,b_j$ quantify a contribution to the differential equation that happens inside the singularity
$\sigma$, and the numbers $d_j$ quantify the interaction of the singularity with $H$ in the vicinity of $\sigma$.

In \cite[\S5]{kaltenbaeck.woracek:p5db} an analogue of the fundamental solution and the monodromy matrix is constructed for indefinite
canonical systems.  Here we only introduce a notation: for an elementary indefinite Hamiltonian $\mf h$ of kind \textup{(A)}
we denote by $W_{\mf h}(t,z)$, $t\in[s_-,\sigma)\cup(\sigma,s_+]$ its \emph{maximal chain of matrices} (the analogue of the fundamental
solution in the positive case), and by $W_{\mf h}$ its \emph{monodromy matrix}. To make the connection to the notation in 
\cite[Theorem~5.1]{kaltenbaeck.woracek:p5db}: there the maximal chain is denoted by $\omega_{\mf h}$ and the monodromy matrix by
$\omega(\mf B(\mf h))$. As already mentioned in the introduction, the monodromy matrix of $\mf h$ (as well as every element of the
maximal chain) is a matrix function in the class $\mc M_{<\infty}$.

\subsection{The operator model}
\label{D157}

Given an elementary indefinite Hamiltonian $\mf h$ of kind (A) one can define an operator model sharing many properties of the
model constructed for a positive Hamiltonian $H\in\HamCC(s_-,s_+)$. It is built from the operator models for $H_\pm$ and an
additional contribution coming from the singularity. This additional contribution is a finite-dimensional part that 
interacts with the $L^2$-spaces surrounding it.
Instead of presenting the original form of the model as constructed in \cite{kaltenbaeck.woracek:p4db}, we give an isomorphic
description similar to \cite[Theorem~2.15]{langer.woracek:esmod}.

Throughout the following we fix an elementary indefinite Hamiltonian $\mf h=\langle H;\oe,b_j;d_j\rangle$ of kind (A) and write 
$\Delta\DE\Delta(H)$.

\begin{Definition}\label{D75}
	Set
	\[
		L^2(H)\DE L^2(H_+)\oplus L^2(H_-), \qquad \Tmax(H)\DE\Tmax(H_+)\oplus\Tmax(H_-)
	\]
	and let $(\omega_n^-)_{n=0}^\infty$ and $(\omega_n^+)_{n=0}^\infty$ be the unique sequences from \cref{D15}
	corresponding to the intervals $(s_-,\sigma)$ and $(\sigma,s_+)$ respectively.
	Moreover, let $\mf w_n$, $n\in\bb N_0$, be the functions defined on $[s_-,\sigma)\cup(\sigma,s_+]$
	whose restrictions to $[s_-,\sigma)$ and $(\sigma,s_+]$ coincide with those from \cref{D17,D18}.
\end{Definition}

\begin{Remark}
\label{D32}
	By the definition of $\Delta$ we know that $\{\mf w_0,\ldots,\mf w_{\Delta-1}\}$ is linearly independent modulo
	$L^2(H)$. Since at least one of the intervals $(s_-,\sigma)$ and $(\sigma,s_+)$ is not indivisible, we obtain from 
	\cite[Lemma~3.11]{kaltenbaeck.woracek:p4db} that $\{\mf w_0,\ldots,\mf w_{\Delta}\}$ is linearly independent modulo
	$\Dom\Tmax(H)$. 
\end{Remark}

\noindent
The model space, which we call $\mrr{\mf P}(\mf h)$, is a space of functions together with a finite-dimensional part.

\begin{Definition}
\label{D33}
	\phantom{}
	\begin{Enumerate}
	\item 
		We set 
		\begin{equation}
		\label{D34}
			L^2_\Delta(H)\DE L^2(H)\dot+\Span\{\mf w_0,\ldots,\mf w_{\Delta-1}\}
		\end{equation}
		and  
		\[
			\mrr{\mf P}(\mf h)\DE L^2_\Delta(H)\times\bb C^{\Delta}\times\bb C^\oe.
		\]
		Elements of $\bb C^{\Delta}$ are generically denoted by $\upxi=(\xi_i)_{i=0}^{\Delta-1}$, and elements of 
		$\bb C^\oe$ by $\upalpha=(\alpha_k)_{k=1}^\oe$.
	\item 
		If $\oe>0$, let $c_1,\ldots,c_\oe$ be the unique numbers such that
		\[
			(c_1,\ldots,c_\oe)
				\begin{pmatrix}
						b_1 & \cdots & b_\oe\\
						 \vdots & \ddots & \vdots \\
						0 & \cdots & b_1
				\end{pmatrix}
			=(-1,0,\ldots,0).
		\]
		Given $(f,\upxi,\upalpha),(g,\upeta,\upbeta)\in\mrr{\mf P}(\mf h)$, we write the unique decompositions of $f$ and $g$ 
		according to the sum and span in \eqref{D34} as 
		\begin{equation}
		\label{D35}
			f=\tilde f+\sum_{i=0}^{\Delta-1}\lambda_i\mf w_i, \qquad 
			g=\tilde g+\sum_{i=0}^{\Delta-1}\mu_i\mf w_i
		\end{equation}
		with $\tilde f,\tilde g\in L^2(H)$ and $\lambda_i,\mu_i\in\bb C$, and define an inner product 
		\[
			\big[(f,\upxi,\upalpha),(g,\upeta,\upbeta)\big] 
			\DE (\tilde f,\tilde g)_H + \sum_{i=0}^{\Delta-1}\lambda_i\ov{\eta_i}
			+ \sum_{i=0}^{\Delta-1}\xi_i\ov{\mu_i} + \sum_{k,l=1}^\oe c_{k+l-\oe}\alpha_k\ov{\beta_l}.
		\]
	\end{Enumerate}
\end{Definition}

\noindent
The space $\mrr{\mf P}(\mf h)$ can be identified with the model space $\mc P(\mf h)$ constructed in 
\cite[\S4.2]{kaltenbaeck.woracek:p4db}.  This is seen using the map $\iota$ from \cite[(4.10)]{kaltenbaeck.woracek:p4db}.

\begin{Definition}
\label{D40}
	We define a map $\mrr\iota\DF\mc P(\mf h)\to\mrr{\mf P}(\mf h)$ as follows. Assume $x\in\mc P(\mf h)$ is given, and write 
	\[
		\iota(x) =
		\Bigl(f,(\xi_i)_{i=0}^{\Delta-1},(\lambda_i)_{i=0}^{\Delta-1},\sum_{k=\Delta}^{\Delta+\oe-1}\alpha_k\delta_k\Bigr).
	\]
	Then we set 
	\[
		\mrr\iota(x) \DE \Bigl(f+\sum_{i=0}^{\Delta-1}\lambda_i\mf w_i,(\xi_i)_{i=0}^{\Delta-1},
		(\alpha_{k-1+\Delta})_{k=1}^\oe\Bigr).
	\]
\end{Definition}

\noindent
By the construction in \cite[pp.~758--760]{kaltenbaeck.woracek:p4db}, the map $\mrr\iota$ is an isometric isomorphism.

Using $\mrr\iota$ we transport the model relation $T(\mf h)$ and the boundary mapping $\Gamma(\mf h)$, which are constructed in
\cite[\S4.2]{kaltenbaeck.woracek:p4db}, to the space $\mrr{\mf P}(\mf h)$. 

\begin{Definition}
\label{D36}
	With the notation from above we set 
	\begin{align*}
		\mrr T(\mf h) &\DE (\mrr\iota\times\mrr\iota)(T(\mf h))\subseteq\mrr{\mf P}(\mf h)\times\mrr{\mf P}(\mf h),
		\\[0.5ex]
		\mrr\Gamma(\mf h) &\DE \Gamma(\mf h)\circ(\mrr\iota\times\mrr\iota)^{-1}|_{\mrr T(\mf h)}
		\DF\mrr T(\mf h)\to\bb C^2\times\bb C^2.
	\end{align*}
\end{Definition}

\noindent
The relation $\mrr T(\mf h)$ and the mapping $\mrr\Gamma(\mf h)$ are finite-dimensional perturbations of $\Tmax(H)$ and 
$\Gamma(H)$ as the following theorem shows.

\begin{theorem}
\label{D38}
	Let $F=(f,\upxi,\upalpha),G=(g,\upeta,\upbeta)\in\mrr{\mf P}(\mf h)$.  Then we have 
	\[
		(F;G) \in \mrr T(\mf h)
	\]
	if and only if the following relations \textup{(i)}--\textup{(v)} hold. 
	\begin{Enumerate}
	\item 
		We have $f=f_0+\sum_{i=0}^\Delta\lambda_i\mf w_i$ with $f_0\in\Dom\Tmax(H)$, 
		$\lambda_0,\ldots,\lambda_\Delta\in\bb C$, and 
		\begin{equation}\label{D73}
			\Big(f_0;g-\sum_{i=0}^{\Delta-1}\lambda_{i+1}\mf w_i\Big)\in\Tmax(H).
		\end{equation}
	\item
		If $\oe=0$, then 
		\begin{align*}
			\xi_{\Delta-1}
			&=  
			\int_{s_-}^{s_+}\mf w_\Delta^*H\Bigl(g-\sum_{i=0}^{\Delta-1}\lambda_{i+1}\mf w_i\Bigr)
			+ \frac 12\sum_{i=0}^{\Delta-1}d_{\Delta-1+i}\lambda_i+d_{2\Delta-1}\lambda_\Delta
			\\
			&\quad +
			\begin{cases}
				-\omega_\Delta^-f(s_-)_1 \CAS (\sigma,s_+)\text{ is indivisible},
				\\[0.5ex]
				\omega_\Delta^+f(s_+)_1 \CAS (s_-,\sigma)\text{ is indivisible},
				\\[0.5ex]
				\omega_\Delta^+f(s_+)_1-\omega_\Delta^-f(s_-)_1 \CASO.
			\end{cases}
		\end{align*}
		If $\oe>0$, then $\alpha_\oe=b_1\lambda_\Delta$.
	\item
		If neither $(s_-,\sigma)$ nor $(\sigma,s_+)$ is indivisible, then 
		\[
			\eta_0=f(s_-)_1-f(s_+)_1-\frac 12\sum_{i=0}^{\Delta-1}d_i\lambda_{i+1}.
		\]
	\item
		For $k\in\{1,\ldots,\Delta-1\}$ we have 
		\begin{align*}
			\eta_k 
			&= \xi_{k-1}-\frac 12 d_{k-1}\lambda_0-\frac 12 d_{k+\Delta-1}\lambda_\Delta
			\\
			&\quad +
			\begin{cases}
				\omega_k^-f(s_-)_1 \CAS (\sigma,s_+)\text{ is indivisible},
				\\[0.5ex]
				-\omega_k^+f(s_+)_1 \CAS (s_-,\sigma)\text{ is indivisible},
				\\[0.5ex]
				\omega_k^-f(s_-)_1-\omega_k^+f(s_+)_1 \CASO.
			\end{cases}
		\end{align*}
	\item
		If $\oe>0$, then 
		\begin{align*}
			\beta_1 &=  
			\int_{s_-}^{s_+}\mf w_\Delta^*H\Big(g-\sum_{i=0}^{\Delta-1}\lambda_{i+1}\mf w_i\Big)
			+ \frac 12\sum_{i=0}^{\Delta-1}d_{\Delta-1+i}\lambda_i+d_{2\Delta-1}\lambda_\Delta-\xi_{\Delta-1}
			\\
			&\quad +
			\begin{cases}
				-\omega_\Delta^-f(s_-)_1 \CAS (\sigma,s_+)\text{ is indivisible},
				\\[0.5ex]
				\omega_\Delta^+f(s_+)_1 \CAS (s_-,\sigma)\text{ is indivisible},
				\\[0.5ex]
				\omega_\Delta^+f(s_+)_1-\omega_\Delta^-f(s_-)_1 \CASO,
			\end{cases}
		\end{align*}
		and 
		\[
			\beta_j=\alpha_{j-1}-b_{\oe+2-j}\lambda_\Delta  \qquad \text{for } j=2,\ldots,\oe.
		\]
	\end{Enumerate}
	Assume that $(F;G)\in\mrr T(\mf h)$.  Then 
	\begin{multline}
		\mrr\Gamma(\mf h)(F;G)(s_-)
		\\
		=
		\begin{cases}
			\begin{pmatrix}f(s_+)_1+\eta_0+\frac 12\sum_{i=0}^{\Delta-1}d_i\lambda_{i+1} \\ \lambda_0\end{pmatrix}
			\CAS (s_-,\sigma)\text{ indivisible},
			\\[6mm]
			f(s_-) \CASO
		\end{cases}
		\label{D60}
	\end{multline}
	and 
	\begin{multline}
		\mrr\Gamma(\mf h)(F;G)(s_+)
		\\
		=
		\begin{cases}
			\begin{pmatrix}f(s_-)_1-\eta_0-\frac 12\sum_{i=0}^{\Delta-1}d_i\lambda_{i+1} \\ \lambda_0\end{pmatrix}
			\CAS (\sigma,s_+)\text{ indivisible},
			\\[6mm]
			f(s_+) \CASO.
		\end{cases}
		\label{D61}
	\end{multline}
\end{theorem}

\begin{Remark}
\label{D39}
	\phantom{}
	\begin{Enumerate}
	\item
		In the space $L^2_\Delta(H)$ we have the natural maximal differential operator
		\begin{align*}
			\TDmax(H) &\DE \Big\{(f,g)\in L_\Delta^2(H) \times L_\Delta^2(H)\DSB 
			\\
			&\hspace*{8ex} \exists\hat f\text{ locally a.c.}\DP \hat f/_{=_H}=f \;\wedge\; \hat f'=JHg\text{ a.e.}\Big\}.
		\end{align*}
		Denote by $\pi$ the projection from $\mrr{\mf P}(\mf h)$ onto $L^2_\Delta(H)$, i.e.\ $\pi((f,\upxi,\upalpha))\DE f$.
		Then we have 
		\[
			(\pi\times\pi)(\mrr T(\mf h))=\TDmax(H).
		\]
		Hence, $\mrr T(\mf h)$ can also be considered as a finite-dimensional perturbation of $\TDmax(H)$
		as can be seen from \eqref{D73}.
		Note that the condition in \Cref{D38}\,(ii) can be seen as a constraint for the domain of $\mrr T(\mf h)$
		whereas (iii)--(v), together with (i), correspond to the action of the operator part.
	\item
		The mapping $\mrr\Gamma(\mf h)\DF\mrr T(\mf h)\to(\bb C^2)^{\{s_-,s_+\}}$ is a boundary mapping 
		so that $\bigl(\mrr{\mf P}(\mf h),\mrr T(\mf h),\mrr\Gamma(\mf h)\bigr)$ becomes a boundary triple
		in the sense of \cite[Definition~2.7]{kaltenbaeck.woracek:p4db};
		note that the boundary mappings have trivial multi-valued part by \cite[Lemma~4.19]{kaltenbaeck.woracek:p4db}.
		In particular, the following abstract Green identity holds: if $(f;g),(u,v)\in\mrr T(\mf h)$, then
		\begin{equation}\label{D77}
		\begin{aligned}
			\relax
			[g,u]-[f,v] &= \mrr\Gamma(\mf h)(u;v)(s_-)^*J\mrr\Gamma(\mf h)(f;g)(s_-)
			\\[0.5ex]
			&\quad - \mrr\Gamma(\mf h)(u;v)(s_+)^*J\mrr\Gamma(\mf h)(f;g)(s_+).
		\end{aligned}
		\end{equation}
		Further, it follows from \cite[Theorem~5.1]{kaltenbaeck.woracek:p4db}
		that the boundary triple has defect 2 and property (E) (see \cite[Definitions~2.8 and 2.16]{kaltenbaeck.woracek:p4db});
		hence, for every $z\in\bb C$ and every $c\in\bb C^2$ there exist unique $F,G\in\ker\bigl(\mrr T(\mf h)-z\bigr)$
		such that $\mrr\Gamma(\mf h)(F;zF)(s_-)=c$ and $\mrr\Gamma(\mf h)(G;zG)(s_+)=c$.
	\end{Enumerate}
\end{Remark}

\noindent
We come to the proof of \cref{D38}.  It is carried out using two ingredients, namely the definition of $T(\mf h)$ as a 
direct sum in \cite[Definition~4.11]{kaltenbaeck.woracek:p4db} and the abstract Green identity \eqref{D77}, where inner products are 
evaluated in $\mrr{\mf P}(\mf h)$.   Here we use, without further comment, the notation from \cite{kaltenbaeck.woracek:p4db}; 
in particular, we set $c_j=0$ for $j\notin\{1,\ldots,\oe\}$ and $\mf b\DE\sum_{k=\Delta}^{\Delta+\oe-1}b_{\Delta+\oe-k}\delta_k$.

\begin{proof}[Proof of \cref{D38}]
	We start with the proof of the forward implication.  Hence, assume that we have an element $(x;y)\in T(\mf h)$ and set 
	\[
		(f,\upxi,\upalpha)\DE\mrr\iota(x), \qquad (g,\upeta,\upbeta)\DE\mrr\iota(y);
	\]
	the task is to prove that \textup{(i)}--\textup{(v)} hold.  On the way we also establish 
	the asserted formulae \eqref{D60}, \eqref{D61} for the boundary values.

	We decompose $(x;y)$ according to \cite[Definition~4.11]{kaltenbaeck.woracek:p4db} as 
	\begin{align}
		\nonumber
		(x;y) &= (x';y')+\kappa_+\big(\chi_+\binom 10;\delta_0\big)+\kappa_-\big(\chi_-\binom 10;-\delta_0\big)
		\\
		\nonumber
		&\quad +\sigma_0(p_0;0)+\sum_{k=1}^{\Delta-1}\sigma_k(p_k;p_{k-1}+d_{k-1}\delta_0)
		+\sigma_\Delta(\mf w_\Delta+\mf b;p_{\Delta-1}+d_{\Delta-1}\delta_0)
		\\
		\label{D37}
		&\quad +\sum_{k=\Delta}^{\Delta+\oe-1}\tau_k(\delta_{k-1};\delta_k).
	\end{align}
	Applying the function $\psi(\mf h)$ we obtain 
	\begin{align*}
		f &= \big[\psi(\mf h)\circ\mrr\iota^{-1}\big](f,\upxi,\upalpha)=\psi(\mf h)(x)
		\\
		&= \big[x'+\kappa_+\chi_+\binom 10+\kappa_-\chi_-\binom 10\big]+\sum_{k=0}^\Delta\sigma_k\mf w_k,
		\\
		g &= \big[\psi(\mf h)\circ\mrr\iota^{-1}\big](g,\upeta,\upbeta)=\psi(\mf h)(y)
		= y'+\sum_{k=1}^\Delta\sigma_k\mf w_{k-1}.
	\end{align*}
	Since 
	\[
		f_0\DE x'+\kappa_+\chi_+\binom 10+\kappa_-\chi_-\binom 10\in\Dom\Tmax(H), 
	\]
	the function $f$ is decomposed as required in (i), namely with the function $f_0$ and the constants 
	$\lambda_i\DE\sigma_i$, $i=0,\ldots,\Delta$. Moreover, since $(x';y')\in\Tmax(H)$, we have 
	\[
		f'=JHg,\quad f_0'=JH\Big(g-\sum_{k=1}^\Delta\sigma_k\mf w_{k-1}\Big).
	\]
	This proves (i).  For later use we also observe that the function $g$ is decomposed as in \eqref{D35} with the constants 
	\[
		\mu_i\DE\sigma_{i+1}=\lambda_{i+1} \qquad\text{for }i=0,\ldots,\Delta-1.
	\]
	Consider now the case when $\oe>0$, so that the parameters $\alpha_k,\beta_k$ are actually present.
	Comparing coefficients we obtain
	\begin{alignat*}{2}
		& \alpha_{k+1-\Delta}=\tau_{k+1}+\sigma_\Delta b_{\Delta+\oe-k} \qquad &&\text{for }k=\Delta,\ldots,\Delta+\oe-2,
		\\[1ex]
		& \alpha_\oe=\sigma_\Delta b_1,
		\\[1ex]
		& \beta_{k+1-\Delta}=\tau_k \qquad &&\text{for }k=\Delta,\ldots,\Delta_\oe-1.
	\end{alignat*}
	It follows that 
	\[
		\beta_j=\alpha_{j-1}-\lambda_\Delta b_{\oe+2-j} \qquad\text{for }j=2,\ldots,\oe.
	\]
	This proves the second formula in (ii) and the second formula in (v).

	To establish the formulae involving $\xi_i$ and $\eta_i$, we apply Green's identity 
	with $(x;y)$ and various other elements in $T(\mf h)$. 
	\begin{Ilist}
	\item 
		With $(p_0;0)$ we obtain:
		\begin{align*}
			& \Gamma(\mf h)(x;y)(s_-)_1-\Gamma(\mf h)(x;y)(s_+)_1
			\\[0.5ex]
			&= -\binom 01^*J\Gamma(\mf h)(x;y)\Big|_{s_-}^{s_+}
			= -\Gamma(\mf h)(p_0;0)^*J\Gamma(\mf h)(x;y)\Big|_{s_-}^{s_+}
			\\[0.5ex]
			&= [y,p_0]-[x,0] = \Bigl[(g,\upeta,\upbeta),\Bigl(\mf w_0,\Bigl(\frac 12d_i\Bigr)_{i=0}^{\Delta-1},0\Bigr)\Bigr]
			\\[0.5ex]
			&= \eta_0+\frac 12\sum_{i=0}^{\Delta-1}d_i\lambda_{i+1}.
		\end{align*}
		If $(s_-,\sigma)$ is not indivisible, then 
		\[
			\Gamma(\mf h)(x;y)(s_-)=f(s_-).
		\]
		Analogously, if $(\sigma,s_+)$ is not indivisible, then 
		\[
			\Gamma(\mf h)(x;y)(s_+)=f(s_+).
		\]
		Hence the relation in (iii) follows.  Moreover, the formula for boundary values in the respective non-indivisible cases
		follows. 

		We also obtain the formula for the upper component of boundary values in the respective indivisible cases. 
		If $(s_-,\sigma)$ is indivisible, then 
		\begin{align*}
			\Gamma(\mf h)(x;y)(s_-)_1 
			&= \Gamma(\mf h)(x;y)(s_+)_1+\eta_0+\frac 12\sum_{i=0}^{\Delta-1}d_i\lambda_{i+1}
			\\
			&= f(s_+)_1+\eta_0+\frac 12\sum_{i=0}^{\Delta-1}d_i\lambda_{i+1}.
		\end{align*}
		Analogously, if $(\sigma,s_+)$ is indivisible, then 
		\begin{align*}
			\Gamma(\mf h)(x;y)(s_+)_1
			&= \Gamma(\mf h)(x;y)(s_-)_1-\eta_0-\frac 12\sum_{i=0}^{\Delta-1}d_i\lambda_{i+1}
			\\
			&= f(s_-)_1-\eta_0-\frac 12\sum_{i=0}^{\Delta-1}d_i\lambda_{i+1}.
		\end{align*}
	\item 
		With $(p_k;p_{k-1}+d_{k-1}\delta_0)$ for $k\in\{1,\ldots,\Delta-1\}$ we obtain:
		\begin{align*}
			& \omega_k^-\Gamma(\mf h)(x;y)(s_-)_1-\omega_k^+\Gamma(\mf h)(x;y)(s_+)_1
			\\[0.5ex]
			&= -\Gamma(\mf h)(p_k;p_{k-1}+d_{k-1}\delta_0)^*J\Gamma(\mf h)(x;y)\Big|_{s_-}^{s_+}
			\\[0.5ex]
			&= [y,p_k]-[x,p_{k-1}+d_{k-1}\delta_0]
			= \Bigl[(g,\upeta,\upbeta),\Bigl(\mf w_k,\Bigl(\frac 12d_{k+i}\Bigr)_{i=0}^{\Delta-1},0\Bigr)\Bigr]
			\\[0.5ex]
			&\quad -\Bigl[(f,\upxi,\upalpha),\Bigl(\mf w_{k-1},\Bigl(\frac 12d_{k-1+i}\Bigr)_{i=0}^{\Delta-1}+
			d_{k-1}\upvarepsilon_0,0\Bigr)\Bigr]
			\\[0.5ex]
			&= \eta_k+\frac 12\sum_{i=0}^{\Delta-1}d_{k+i}\lambda_{i+1}-\xi_{k-1}
			-\frac 12\sum_{i=0}^{\Delta-1}d_{k-1+i}\lambda_i+d_{k-1}\lambda_0
			\\[0.5ex]
			&= \eta_k+\frac 12d_{k-1+\Delta}\lambda_\Delta-\xi_{k-1}+\frac 12d_{k-1}\lambda_0.
		\end{align*}
		This yields (iv).  Remember here that $\omega_k^-=0$ if $(s_-,\sigma)$ is indivisible, and that 
		$\omega_k^+=0$ if $(\sigma,s_+)$ is indivisible. 
	\item 
		With $(\mf w_\Delta+\mf b;p_{\Delta-1}+d_{\Delta-1}\delta_0)$ we obtain:
		\begin{align*}
			& \omega_\Delta^-\Gamma(\mf h)(x;y)(s_-)_1-\omega_\Delta^+\Gamma(\mf h)(x;y)(s_+)_1
			\\[0.5ex]
			&= -\Gamma(\mf h)(\mf w_\Delta+\mf b;p_{\Delta-1}+d_{\Delta-1}\delta_0)^*J
			\Gamma(\mf h)(x;y)\Big|_{s_-}^{s_+}
			\\[0.5ex]
			&= [y,\mf w_\Delta+\mf b]-[x,p_{\Delta-1}+d_{\Delta-1}\delta_0]
			\\[0.5ex]
			&= \Bigl[(g,\upeta,\upbeta),\Bigl(\mf w_\Delta,\Bigl(\frac 12d_{\Delta+i}\Bigr)_{i=0}^{\Delta-1},
			(b_{\oe+1-k})_{k=1}^\oe\Bigr)\Bigr]
			\\[0.5ex]
			&\quad -\Bigl[(f,\upxi,\upalpha),(\mf w_{k-1},\Bigl(\frac 12d_{k-1+i}\Bigr)_{i=0}^{\Delta-1}+d_{k-1}\upvarepsilon_0,0\Bigr)\Bigr]
			\\[0.5ex]
			&= \Big(g-\sum_{i=0}^{\Delta-1}\lambda_{i+1}\mf w_i,\mf w_\Delta\Big)_H
			+\frac 12\sum_{i=0}^{\Delta-1}d_{\Delta+i}\lambda_{i+1}
			\\[0.5ex]
			&\quad +
			\begin{cases}
				0 \CAS \oe=0,
				\\
				\sum_{k,l=1}^\oe c_{k+l-\oe}\beta_k b_{\oe+1-l} \CAS \oe>0,
			\end{cases}
			\\[0.5ex]
			&\quad -\xi_{\Delta-1}-\frac 12\sum_{i=0}^{\Delta-1}d_{\Delta-1+i}\lambda_i+d_{\Delta-1}\lambda_0
			\\[0.5ex]
			&= \int\limits_{s_-}^{s_+}\mf w_\Delta^*H\Big(g-\sum_{i=0}^{\Delta-1}\lambda_{i+1}\mf w_i\Big)
			+\frac 12\sum_{i=0}^{\Delta-1}d_{\Delta-1+i}\lambda_i+d_{2\Delta-1}\lambda_\Delta-\xi_{\Delta-1}
			\\[0.5ex]
			&\quad +
			\begin{cases}
				0 \CAS \oe=0,
				\\
				-\beta_1 \CAS \oe>0.
			\end{cases}
		\end{align*}
		Note here that, by the definition of $c_j$, we have 
		\[
			\sum_{l=1}^\oe c_{k+l-\oe}b_{\oe+1-l}=\sum_{j=1}^k c_jb_{k+1-j}=
			\begin{cases}
				-1 \CAS k=1,
				\\
				0 \CAS k>1.
			\end{cases}
		\]
		This shows the first formula in (ii) and the first formula in (v).
	\item 
		Assume that $(s_-,\sigma)$ is indivisible and use $(0;-\delta_0)$; note here that 
		$(0;-\delta_0)=(\chi_-\binom 10;-\delta_0)$:
		\[
			\Gamma(\mf h)(x;y)(s_-)_2
			= -\Gamma(\mf h)(0;-\delta_0)^*J\Gamma(\mf h)(x;y)\Big|_{s_-}^{s_+}
			= [y,0]-[x,-\delta_0] = \lambda_0.
		\]
		Assume that $(\sigma,s_+)$ is indivisible and use $(0;\delta_0)$; now 
		$(0;\delta_0)=(\chi_+\binom 10;\delta_0)$:
		\[
			-\Gamma(\mf h)(x;y)(s_+)_2
			= -\Gamma(\mf h)(0;-\delta_0)^*J\Gamma(\mf h)(x;y)\Big|_{s_-}^{s_+}
			= [y,0]-[x,\delta_0] = -\lambda_0.
		\]
		This proves the assertion about the second component of the boundary values. 
	\end{Ilist}
	The proof of the forward implication is finished, and the asserted formulae for the boundary values are established.

	For the proof of the converse implication assume that we have a pair 
	\[
		\big((f,\upxi,\upalpha),(g,\upeta,\upbeta)\big)\in\mrr{\mf P}(\mf h)\times\mrr{\mf P}(\mf h)
	\]
	that satisfies (i)--(v).  Set 
	\[
		\kappa_+\DE\Gamma(H)\Big(f_0;g-\sum_{i=0}^{\Delta-1}\lambda_{i+1}\mf w_i\Big)(s_+)_1,\quad 
		\kappa_-\DE\Gamma(H)\Big(f_0;g-\sum_{i=0}^{\Delta-1}\lambda_{i+1}\mf w_i\Big)(s_-)_1.
	\]
	Then
	\[
		\Big(f_0-\kappa_+\chi_+\binom 10-\kappa_-\chi_-\binom 10;g-\sum_{i=0}^{\Delta-1}\lambda_{i+1}\mf w_i\Big)
		\in B^{-1}.
	\]
	Choose $(x';y')$ in the set \cite[(4.13)]{kaltenbaeck.woracek:p4db} such that 
	\[
		[\psi(\mf h)\times\psi(\mf h)](x';y')=
		\Big(f_0-\kappa_+\chi_+\binom 10-\kappa_-\chi_-\binom 10;g-\sum_{i=0}^{\Delta-1}\lambda_{i+1}\mf w_i\Big).
	\]
	Set 
	\begin{alignat*}{2}
		& \sigma_i\DE\lambda_i \qquad &&\text{for }i=0,\ldots,\Delta,
		\\[1ex]
		& \tau_k\DE\beta_{k-\Delta+1} \qquad &&\text{for }k=\Delta,\ldots,\Delta+\oe-1,
	\end{alignat*}
	and 
	\begin{align*}
		(x;y) &\DE (x';y')+\kappa_+\Bigl(\chi_+\binom 10;\delta_0\Bigr)+\kappa_-\Bigl(\chi_-\binom 10;-\delta_0\Bigr)
		\\
		&\quad +\sigma_0(p_0;0)+\sum_{k=1}^{\Delta-1}\sigma_k(p_k;p_{k-1}+d_{k-1}\delta_0)
		+\sigma_\Delta(\mf w_\Delta+\mf b;p_{\Delta-1}+d_{\Delta-1}\delta_0)
		\\
		&\quad +\sum_{k=\Delta}^{\Delta+\oe-1}\tau_k(\delta_{k-1};\delta_k).
	\end{align*}
	Then $(x;y)\in T(\mf h)$.  Further, set 
	\[
		(\tilde f,\tilde\upxi,\tilde\upalpha)\DE\mrr\iota(x), \qquad 
		(\tilde g,\tilde\upeta,\tilde\upbeta)\DE\mrr\iota(y).
	\]
	By the already proved forward implication the pair $(\mrr\iota(x);\mrr\iota(y))$ belongs to $\mrr T(\mf h)$ and 
	satisfies (i)--(v).  We use this knowledge to show that 
	\[
		\big((\tilde f,\tilde\upxi,\tilde\upalpha),(\tilde g,\tilde\upeta,\tilde\upbeta)\big)-
		\big((f,\upxi,\upalpha),(g,\upeta,\upbeta)\big)\in\mrr T(\mf h).
	\]
	Again we proceed in a couple of steps.
	\begin{Ilist}
	\item 
		We evaluate the function part.
		\begin{align*}
			\tilde f= &\, \psi(\mf h)(x)
			\\
			= &\, \big[f_0-\kappa_+\chi_+\binom 10-\kappa_-\chi_-\binom 10\big]
			+\kappa_+\chi_+\binom 10-\kappa_-\chi_-\binom 10+\sum_{k=0}^\Delta\sigma_k\mf w_k
			\\
			= &\, f_0+\sum_{k=0}^\Delta\sigma_k\mf w_k=f.
		\end{align*}
		In particular, in the decomposition of $\tilde f$ according to (i) (let the constants corresponding to 
		$\tilde f$ be denoted by $\tilde\lambda_i$) we have $\tilde\lambda_i=\lambda_i$ for 
		$i=0,\ldots,\Delta$. 

		Applying $\psi(\mf h)$ to the second component yields 
		\[
			\tilde g=\psi(\mf h)(y)=\Big[g-\sum_{i=0}^{\Delta-1}\lambda_{i+1}\mf w_i\Big]
			+\sum_{k=1}^\Delta\sigma_k\mf w_{k-1}=g.
		\]
		If $(s_-,\sigma)$ is not indivisible, then $\tilde f(s_-)=f(s_-)$.  Analogously, if $(\sigma,s_+)$ is not 
		indivisible, then $\tilde f(s_+)=f(s_+)$.
	\item 
		Let $k\in\{1,\ldots,\Delta-1\}$; then 
		\begin{align*}
			& \tilde\eta_k-\tilde\xi_{k-1}
			\\[1ex]
			&= -\frac 12d_{k-1}\tilde\lambda_0-d_{k-1+\Delta}\tilde\lambda_0+
			{\small
			\begin{cases}
				\omega_k^-f(s_-) \CAS (\sigma,s_+)\text{ indivisible},
				\\[0.5ex]
				-\omega_k^+f(s_+) \CAS (s_-,\sigma)\text{ indivisible},
				\\[0.5ex]
				\omega_k^-f(s_-)-\omega_k^+f(s_+) \CASO,
			\end{cases}
			}
			\\[1ex]
			&=  
			-\frac 12d_{k-1}\lambda_0-d_{k-1+\Delta}\lambda_0+
			{\small
			\begin{cases}
				\omega_k^-f(s_-) \CAS (\sigma,s_+)\text{ indivisible},
				\\[0.5ex]
				-\omega_k^+f(s_+) \CAS (s_-,\sigma)\text{ indivisible},
				\\[0.5ex]
				\omega_k^-f(s_-)-\omega_k^+f(s_+) \CASO,
			\end{cases}
			}
			\\
			&= \eta_k-\xi_{k-1}.
		\end{align*}
	\item 
		Assume that $\oe=0$. Then, in the same way as above, we obtain that 
		\[
			\tilde\xi_{\Delta-1}=\xi_{\Delta-1}.
		\]
		If neither $(s_-,\sigma)$ nor $(\sigma,s_+)$ is indivisible, then also 
		\[
			\tilde\eta_0=\eta_0,
		\]
		and hence 
		\begin{multline*}
			\big((\tilde f,\tilde\upxi,\tilde\upalpha),(\tilde g,\tilde\upeta,\tilde\upbeta)\big)
			- \big((f,\upxi,\upalpha),(g,\upeta,\upbeta)\big)
			\\
			\in(\mrr\iota\times\mrr\iota)
			\Big(\Span\big\{(\delta_{k-1};\delta_k)\DS k=1,\ldots,\Delta-1\big\}\Big)
			\subseteq\mrr T(\mf h).
		\end{multline*}
		If one of $(s_-,\sigma)$ and $(\sigma,s_+)$ is indivisible, then 
		\begin{multline*}
			\big((\tilde f,\tilde\upxi,\tilde\upalpha),(\tilde g,\tilde\upeta,\tilde\upbeta)\big)
			-\big((f,\upxi,\upalpha),(g,\upeta,\upbeta)\big)
			\\
			\in(\mrr\iota\times\mrr\iota)
			\Big(
			\Span\big(\{(0;\delta_0)\}\cup\big\{(\delta_{k-1};\delta_k)\DS k=1,\ldots,\Delta-1\big\}\big)
			\Big)
			\subseteq\mrr T(\mf h).
		\end{multline*}
	\item
		Assume that $\oe>0$.  Then 
		\begin{align*}
			& \tilde\beta_j=\tau_{j+\Delta-1}=\beta_j \qquad\text{for }j=1,\ldots,\oe,
			\\
			& \tilde\alpha_\oe=\tilde\lambda_\Delta b_1=\lambda_\Delta b_1=\alpha_\oe.
		\end{align*}
		Using $\tilde\beta_j=\beta_j$ for $j=2,\ldots,\oe$ we obtain 
		\[
			\tilde\alpha_j=\tilde\beta_{j+1}+\tilde\lambda_\Delta b_{\oe+1-j}=\beta_{j+1}+\lambda_\Delta b_{\oe+1-j}
			=\alpha_j \qquad\text{for }j=1,\ldots,\oe-1.
		\]
		Using $\tilde\beta_1=\beta_1$ we obtain that 
		\[
			\tilde\xi_{\Delta-1}=\xi_{\Delta-1}.
		\]
		If neither $(s_-,\sigma)$ nor $(\sigma,s_+)$ is indivisible, then also 
		\[
			\tilde\eta_0=\eta_0,
		\]
		and hence 
		\begin{multline*}
			\big((\tilde f,\tilde\upxi,\tilde\upalpha),(\tilde g,\tilde\upeta,\tilde\upbeta)\big)
			- \big((f,\upxi,\upalpha),(g,\upeta,\upbeta)\big)
			\\
			\in(\mrr\iota\times\mrr\iota)
			\Big(\Span\big\{(\delta_{k-1};\delta_k)\DS k=1,\ldots,\Delta+\oe-1\big\}\Big)
			\subseteq\mrr T(\mf h).
		\end{multline*}
		If one of $(s_-,\sigma)$ and $(\sigma,s_+)$ is indivisible, then 
		\begin{multline*}
			\big((\tilde f,\tilde\upxi,\tilde\upalpha),(\tilde g,\tilde\upeta,\tilde\upbeta)\big)
			- \big((f,\upxi,\upalpha),(g,\upeta,\upbeta)\big)
			\\
			\in(\mrr\iota\times\mrr\iota)
			\Big(
			\Span\big(\{(0;\delta_0)\}\cup\big\{(\delta_{k-1};\delta_k)\DS k=1,\ldots,\Delta+\oe-1\big\}\big)
			\Big)
			\subseteq\mrr T(\mf h).
		\end{multline*}
	\end{Ilist}
	This finishes the proof of the backward implication.
\end{proof}

%
\section{Solution using regularised boundary values}
\label{D158}
%

%
\subsection{Regularised boundary values}
\label{D159}

We encode the discrete data of $\mf h$ in a polynomial. 

\begin{Definition}\label{D185}
	Let $\mf h=\langle H;\oe,b_j;d_j\rangle$ be an elementary indefinite Hamiltonian of kind (A)
	and define the polynomial
	\begin{align}\label{D176}
		\ms p(z) \DE&\; -\sum_{n=1}^{2\Delta}d_{n-1}z^n + \sum_{n=2\Delta+1}^{2\Delta+\oe}b_{\oe+2\Delta+1-n}z^n
		\\[1ex]
		=&\; -d_0z-\ldots-d_{2\Delta-1}z^{2\Delta} +
		\begin{cases}
			0 & \text{if} \ \oe=0,
			\\[0.5ex]
			b_{\oe}z^{2\Delta+1}+\ldots+b_1z^{2\Delta+\oe} &\text{if} \ \oe>0.
		\end{cases}
		\nonumber	
	\end{align}
\end{Definition}

\noindent
In the following we often write $\mrr\Gamma$ for $\mrr\Gamma(\mf h)$.

\begin{theorem}\label{D115}
	Let $\mf h=\langle H;\oe,b_j;d_j\rangle$ be an elementary indefinite Hamiltonian of kind \textup{(A)}, let $z\in\bb C$, 
	let $F=(f;\upxi,\upalpha)\in\mrr{\mf P}(\mf h)$ such that $(F;zF)\in\mrr T(\mf h)$  
	and let $\hat f=(\hat f_-,\hat f_+)$ be a locally absolutely continuous representative of $f$ on $(s_-,\sigma)\cup(\sigma,s_+)$.  
	Then the limits 
	\begin{align}
		\label{D134}
		\Gammar^\pm \hat f_\pm &\DE \lim_{x\to\sigma\pm}\hat f_\pm(x)_2,
		\\[1ex]
		\label{D135}
		\Gammas^\pm(z) \hat f_\pm &\DE \lim_{x\to\sigma\pm}
		\sum_{n=0}^\Delta z^n\bigl(\mf w_n(x)\bigr)^*J
		\biggl(\hat f_\pm(x)-\bigl(\Gammar^\pm \hat f_\pm\bigr)\sum_{j=\Delta+1}^{2\Delta-n}z^j\mf w_j(x)\biggr)
	\end{align}
	exists and
	\begin{align}
		\Gammar^+\hat f_+ &= \Gammar^-\hat f_-,
		\label{D136}
		\\[1ex]
		\Gammas^+(z)\hat f_+ &= \Gammas^-(z)\hat f_- + \ms p(z)\Gammar^-\hat f_- 
		\nonumber\\[0.5ex]
		&\quad + \mrr\Gamma(F;zF)(s_-)_1 - \hat f_-(s_-)_1 - \mrr\Gamma(F;zF)(s_+)_1 + \hat f_+(s_+)_1.
		\label{D137}
	\end{align}
	Moreover, the following statements are true.
	\begin{Enumerate}
	\item
		If $(s_-,\sigma)$ is not indivisible, then $\hat f_-$ is uniquely determined and 
		\begin{equation}\label{D130}
			\hat f_-(s_-) = \mrr\Gamma(F,zF)(s_-).
		\end{equation}
		Otherwise, $\hat f_-$ is unique up to an additive multiple of $\binom10$, and there exists exactly one $\hat f_-$
		so that \eqref{D130} holds.
	\item
		If $(\sigma,s_+)$ is not indivisible, then $\hat f_+$ is uniquely determined and
		\begin{equation}\label{D131}
			\hat f_+(s_+) = \mrr\Gamma(F,zF)(s_+).
		\end{equation}
		Otherwise, $\hat f_+$ is unique up to an additive multiple of $\binom10$, and there exists exactly one $\hat f_+$
		so that \eqref{D131} holds.
	\end{Enumerate}
\end{theorem}

\medskip

\noindent
Before we prove \cref{D115}, let us add some remarks and introduce some notation.

\pagebreak[3]

\begin{remark}\label{D119}
\rule{0ex}{1ex}
\begin{Enumerate}
\item
	If $(s_-,\sigma)$ and $(\sigma,s_+)$ are both not indivisible, the relation \cref{D137} automatically becomes 
	\begin{equation}\label{D140}
		\Gammas^+(z)\hat f_+ = \Gammas^-(z)\hat f_- + \ms p(z)\Gammar^-\hat f_-.
	\end{equation}
	In general the choice of $\hat f_+$ or $\hat f_-$ can always be made such that we have \cref{D140} instead of \cref{D137}.
\item
	The generalised boundary mapping $\Gammas^-(z)$ depends only on $H_-$ and $\Delta$;
	the latter depends on the strengths of the singularities of both $H_-$ and $H_+$.
\item
	If $H_-$ is diagonal, then $\mf w_{2\Delta}$ is not needed for the evaluation of $\Gammas^-(z)\hat f_-$
	since $(\mf w_0(x))^*J\mf w_{2\Delta}(x)=0$ by \cref{D63}\,(ii).
	A similar statement holds for $H_+$.
\end{Enumerate}
\end{remark}

\begin{Definition}\label{D170}
	We set
	\begin{equation}\label{D139}
		\Gamma^\pm(z)\hat f \DE \begin{pmatrix} \Gammas^\pm(z)\hat f \\[1ex] \Gammar^\pm\hat f \end{pmatrix}
	\end{equation}
	for $\hat f$ as in \cref{D115}
	and define the matrix
	\begin{equation}\label{D189}
		\ms R(z) \DE \begin{pmatrix} 1 & \ms p(z) \\[1ex] 0 & 1 \end{pmatrix}.
	\end{equation}
\end{Definition}

\begin{Remark}\label{D138}
	With the notation from \cref{D170} the relations in \eqref{D136} and \eqref{D140} can now be written as
	\begin{equation}\label{D177}
		\Gamma^+(z)\hat f = \ms R(z)\Gamma^-(z)\hat f.
	\end{equation}
\end{Remark}

\medskip

\noindent
For the proof of \cref{D115} let $F$ be as in \cref{D115}.  Since $(F;zF)\in\mrr T(\mf h)$, we can write
\[
	f = f_0 + \sum_{l=0}^\Delta \lambda_l\mf w_l = \tilde f + \sum_{l=0}^{\Delta-1} \lambda_l\mf w_l
\]
with $f_0\in\Tmax(H)$ and $\tilde f\in L^2(H)$ by (i) in \cref{D38}.
Further, set 
\begin{equation}\label{D172}
	G\DE zF = (g;\upeta,\upbeta).
\end{equation}
First we need a couple of lemmas.

\begin{lemma}\label{D120}
	Let $F$ and $\hat f$ be as in \cref{D115} and let $G$ be as in \eqref{D172}.  Then
	\begin{align}
		\lambda_k &= z^k\lambda_0,
		\label{D117}
		\\[1ex]
		\xi_k &= z^{\Delta-k}\int_{s_-}^{s_+} \mf w_\Delta^* H\Bigl(f-\sum_{l=0}^{\Delta-1}z^l\lambda_0\mf w_l\Bigr)
		+ \sum_{l=1}^{\Delta-k} z^{l-1}
		\bigl(\omega_{k+l}^+\hat f(s_+)_1 - \omega_{k+l}^-\hat f(s_-)_1\bigr)
		\nonumber\\[1ex]
		&\quad + \Biggl(\frac{1}{2}\sum_{l=0}^{\Delta-1} d_{k+l}z^l
		+ \sum_{l=0}^{\Delta-k-1}d_{\Delta+l+k}z^{\Delta+l}
		- \sum_{l=1}^\oe b_{\oe+1-l}z^{2\Delta+l-k-1}\Biggr)\lambda_0
		\label{D105}
	\end{align}
	for $k\in\{0,\ldots,\Delta-1\}$.
	If $\oe>0$, then
	\begin{equation}\label{D118}
		\alpha_{\oe-k} = \sum_{j=0}^k b_{k+1-j}\,z^{\Delta+j}\lambda_0,
		\qquad k\in\{0,\ldots,\oe-1\}.
	\end{equation}
\end{lemma}

\begin{proof}
	We have $g=zf$ by \eqref{D172}, and hence
	\begin{equation}\label{D173}
		g - \sum_{l=0}^{\Delta-1}\lambda_{l+1}\mf w_l
		= zf - \sum_{l=0}^{\Delta-1}\lambda_{l+1}\mf w_l
		= zf_0 + z\lambda_\Delta\mf w_\Delta + \sum_{l=0}^{\Delta-1}(z\lambda_l-\lambda_{l+1})\mf w_l.
	\end{equation}
	Since the left-hand side is in $L^2(H)$ by (i) in \cref{D38}, it follows that $\lambda_{l+1}=z\lambda_l$
	for $l\in\{0,\ldots,\Delta-1\}$, which yields \eqref{D117}.
	
	In the case when $\oe>0$ we obtain from (v) and (ii) in \cref{D38} that
	\begin{align*}
		\alpha_{j-1} &= z\alpha_j + b_{\oe+2-j}z^\Delta\lambda_0, \qquad j=2,\ldots,\oe,
		\\[1ex]
		\alpha_{\oe} &= b_1 z^\Delta\lambda_0,
	\end{align*}
	which, by induction, implies \eqref{D118},
	and, in particular,
	\[
		\beta_1 = z\alpha_1
		= z\sum_{j=0}^{\oe-1} b_{\oe-j}\,z^{\Delta+j}\lambda_0
		= \sum_{l=1}^\oe b_{\oe+1-l}\,z^{\Delta+l}\lambda_0.
	\]

	Next we prove \eqref{D105} by induction.
	Let us start with $k=\Delta-1$: when $\oe=0$, we use (ii) in \cref{D38}; when $\oe>0$, we use (v) in \cref{D38} to obtain
	\begin{align*}
		\xi_{\Delta-1} &= \int_{s_-}^{s_+} \mf w_\Delta^* H\Bigl(zf-\sum_{l=0}^{\Delta-1}\lambda_{l+1}\mf w_l\Bigr)
		+ \frac{1}{2}\sum_{l=0}^{\Delta-1} d_{l+\Delta-1}z^l \lambda_0
		+ d_{2\Delta-1}z^\Delta \lambda_0
		\\[1ex]
		&\quad + 
		\begin{cases}
			- \omega_\Delta^- f(s_-)_1 & \text{if $(\sigma,s_+)$ is indivisible},
			\\[0.5ex]
			\omega_\Delta^+ f(s_+)_1 & \text{if $(s_-,\sigma)$ is indivisible},
			\\[0.5ex]
		 	\omega_\Delta^+ f(s_+)_1 - \omega_\Delta^- f(s_-)_1 & \text{otherwise}
		\end{cases}
		\\[1ex]
		&\quad - 
		\begin{cases}
			\beta_1 &\text{if} \ \oe>0,
			\\[0.5ex]
			0 &\text{if} \ \oe=0
		\end{cases}
		\\[1ex]
		&= z\int_{s_-}^{s_+} \mf w_\Delta^* H\Bigl(f-\sum_{l=0}^{\Delta-1}z^l\lambda_0\mf w_l\Bigr)
		+ \omega_\Delta^+\hat f(s_+)_1 - \omega_\Delta^-\hat f(s_-)_1
		\\[1ex]
		&\quad + \Biggl(\frac{1}{2}\sum_{l=0}^{\Delta-1} d_{l+\Delta-1}z^l
		+ d_{2\Delta-1}z^\Delta - \sum_{l=1}^\oe b_{\oe+1-l}z^{\Delta+l}\Biggr)\lambda_0;
	\end{align*}
	note that $\omega_\Delta^-=0$ if $(s_-,\sigma)$ is indivisible and that $\hat f_-=f_-$ otherwise,
	and a similar statement is true for $(\sigma,s_+)$.
	This proves \eqref{D105} for $k=\Delta-1$.

	Now let $k\in\{0,\ldots,\Delta-2\}$ and assume that \eqref{D105} is true
	with $k$ replaced by $k+1$.
	Then (iv) in \cref{D38} (with a similar consideration as above using $\omega_{k+1}^-=0$ if $(s_-,\sigma)$ is indivisible 
	and $\hat f_-=f_-$ otherwise) implies that
	\begin{align*}
		\xi_k &= z\xi_{k+1} + \frac{1}{2}d_{\Delta+k}z^\Delta\lambda_0
		+ \frac{1}{2}d_k\lambda_0
		+ \omega_{k+1}^+\hat f(s_+)_1 - \omega_{k+1}^-\hat f(s_-)_1
		\\[1ex] 
		&= zz^{\Delta-k-1}\int_{s_-}^{s_+} \mf w_\Delta^* H\Bigl(f-\sum_{l=0}^{\Delta-1}z^l\lambda_0\mf w_l\Bigr)
		\\[1ex]
		&\quad + z\sum_{l=1}^{\Delta-k-1}z^{l-1}\bigl(\omega_{k+l+1}^+\hat f(s_+)_1
		- \omega_{k+l+1}^-\hat f(s_-)_1\bigr)
		\\[1ex]
		&\quad + z\Biggl(\frac{1}{2}\sum_{l=0}^{\Delta-1} d_{k+l+1}z^l
		+ \sum_{l=0}^{\Delta-k-2}d_{\Delta+l+k+1}z^{\Delta+l}
		- \sum_{l=1}^\oe b_{\oe+1-l}z^{2\Delta+l-k-2}\Biggr)\lambda_0
		\\[1ex]
		&\quad + \frac{1}{2}d_{\Delta+k}z^\Delta\lambda_0 + \frac{1}{2}d_k\lambda_0
		+ \omega_{k+1}^+\hat f(s_+)_1 - \omega_{k+1}^-\hat f(s_-)_1
		\displaybreak[0]\\[1ex] 
		&= z^{\Delta-k}\int_{s_-}^{s_+} \mf w_\Delta^* H\Bigl(f-\sum_{l=0}^{\Delta-1}z^l\lambda_0\mf w_l\Bigr)
		+ \sum_{l=2}^{\Delta-k} z^{l-1}\bigl(\omega_{k+l}^+\hat f(s_+)_1 - \omega_{k+l}^-\hat f(s_-)_1\bigr)
		\\[1ex]
		&\quad + \omega_{k+1}^+\hat f(s_+)_1 - \omega_{k+1}^-\hat f(s_-)_1
		\\[1ex]
		&\quad + \Biggl(\frac{1}{2}\sum_{l=1}^\Delta d_{k+l}z^l
		+ \sum_{l=1}^{\Delta-k-1} d_{\Delta+l+k}z^{\Delta+l}
		+ \frac{1}{2}d_{\Delta+k}z^\Delta + \frac{1}{2}d_k
		\\[1ex]
		&\quad - \sum_{l=1}^\oe b_{\oe+1-l}z^{2\Delta+l-k-1}\Biggr)\lambda_0
		\displaybreak[0]\\[1ex] 
		&= z^{\Delta-k}\int_{s_-}^{s_+} \mf w_\Delta^* H\Bigl(f-\sum_{l=0}^{\Delta-1}z^l\lambda_0\mf w_l\Bigr)
		+ \sum_{l=1}^{\Delta-k} z^{l-1}\bigl(\omega_{k+l}^+\hat f(s_+)_1 - \omega_{k+l}^-\hat f(s_-)_1\bigr)
		\\[1ex]
		&\quad + \Biggl(\frac{1}{2}\sum_{l=0}^{\Delta-1} d_{k+l}z^l
		+ \sum_{l=0}^{\Delta-k-1}d_{\Delta+l+k}z^{\Delta+l}
		- \sum_{l=1}^\oe b_{\oe+1-l}z^{2\Delta+l-k-1}\Biggr)\lambda_0,
	\end{align*}
	which is equal to the right-hand side of \eqref{D105}.
	Hence \eqref{D105} holds for all $k\in\{0,\ldots,\Delta-1\}$.
\end{proof}

\begin{lemma}\label{D121}
	Under the assumptions of \cref{D115} we have
	\begin{equation}\label{D107}
	\begin{aligned}
		& z^{\Delta+1}\int_{s_-}^{s_+} \mf w_\Delta^* H\Bigl(f-\sum_{l=0}^{\Delta-1}z^l\lambda_0\mf w_l\Bigr)
		+ \sum_{l=1}^\Delta z^l\bigl(\omega_l^+\hat f(s_+)_1 - \omega_l^-\hat f(s_-)_1\bigr)
		\\[1ex]
		&= \lambda_0\ms p(z) + \mrr\Gamma(F;zF)(s_-)_1 - \mrr\Gamma(F;zF)(s_+)_1.
	\end{aligned}
	\end{equation}
\end{lemma}

\begin{proof}
%
	If neither $(s_-,\sigma)$ nor $(\sigma,s_+)$ is indivisible we use (iii) in \cref{D38}; 
	if $(s_-,\sigma)$ or $(\sigma,s_+)$ is indivisible, we use \eqref{D60} or \eqref{D61} respectively 
	to obtain
	\begin{align}
		\eta_0 &= -\frac{1}{2}\sum_{l=0}^{\Delta-1} d_l\lambda_{l+1}
		\nonumber\\[1ex]
		&\quad +
		\begin{cases}
			\hat f(s_-)_1 - \hat f(s_+)_1 & \text{if neither $(s_-,\sigma)$ nor $(\sigma,s_+)$ is indiv.,}
			\\[1ex]
			\mrr\Gamma(F;zF)(s_-)_1 - \hat f(s_-)_1 & \text{if $(s_-,\sigma)$ is indivisible,}
			\\[1ex]
			-\mrr\Gamma(F;zF)(s_+)_1 + \hat f(s_+)_1 & \text{if $(\sigma,s_+)$ is indivisible}
		\end{cases}
		\nonumber\\[1ex]
		&= -\frac{1}{2}\sum_{l=0}^{\Delta-1} d_l z^{l+1}\lambda_0 + \mrr\Gamma(F;zF)(s_-)_1 - \mrr\Gamma(F;zF)(s_+)_1,
		\label{D106}
	\end{align}
	where, for the last relation, we used again \eqref{D60} and \eqref{D61}, namely, 
	e.g.\ $\mrr\Gamma(F;zF)(s_-)_1=\hat f(s_-)_1$ if $(s_-,\sigma)$ is not indivisible.
	On the other hand, \eqref{D105} for $k=0$ yields
	\begin{align*}
		\eta_0 &= z\xi_0
		\notag\\[1ex]
		&= z^{\Delta+1}\int_{s_-}^{s_+} \mf w_\Delta^* H\Bigl(f-\sum_{l=0}^{\Delta-1}z^l\lambda_0\mf w_l\Bigr)
		+ \sum_{l=1}^\Delta z^l\bigl(\omega_l^+\hat f(s_+)_1 - \omega_l^-\hat f(s_-)_1\bigr)
		\\[1ex]
		&\quad + \Biggl(\frac{1}{2}\sum_{l=0}^{\Delta-1} d_l z^{l+1}
		+ \sum_{l=0}^{\Delta-1} d_{\Delta+l}z^{\Delta+l+1}
		- \sum_{l=1}^\oe b_{\oe+1-l}z^{2\Delta+l}\Biggr)\lambda_0.
	\end{align*}
	Together with \eqref{D106}, this implies \eqref{D107}.
\end{proof}

\begin{lemma}\label{D122}
	For $x\in(s_-,\sigma)$ we have
	\begin{align}
		&\hspace*{-5ex} z^{\Delta+1}\int_{s_-}^x \mf w_\Delta^*H\Bigl(f-\sum_{l=0}^{\Delta-1}z^l\lambda_0\mf w_l\Bigr)
		- \sum_{l=0}^{\Delta}z^l\omega_l^-\hat f(s_-)_1
		\nonumber\\[1ex]
		&= - \sum_{n=0}^\Delta z^n\bigl(\mf w_n(x)\bigr)^*J\biggl(\hat f(x)-\lambda_0\sum_{j=\Delta+1}^{2\Delta-n}z^j\mf w_j(x)\biggr),
		\label{D113}
	\end{align}
	and for $x\in(\sigma,s_+)$ we have
	\begin{align}
		&\hspace*{-5ex} z^{\Delta+1}\int_x^{s_+} \mf w_\Delta^*H\Bigl(f-\sum_{l=0}^{\Delta-1}z^l\lambda_0\mf w_l\Bigr)
		+ \sum_{l=0}^{\Delta}z^l\omega_l^+\hat f(s_+)_1
		\nonumber\\[1ex]
		&= \sum_{n=0}^\Delta z^n\bigl(\mf w_n(x)\bigr)^*J\biggl(\hat f(x)-\lambda_0\sum_{j=\Delta+1}^{2\Delta-n}z^j\mf w_j(x)\biggr).
		\label{D114}
	\end{align}
\end{lemma}

\begin{proof}
	We only prove \eqref{D113}; the proof of \eqref{D114} is similar.
	Take an arbitrary $x\in(s_-,\sigma)$.
	First note that we can use the representative $\hat f$ instead of $f$ in the integral on the left-hand side of \eqref{D113}.
	Using induction we prove the following relation for $k\in\{0,\ldots,\Delta+1\}$:
	\begin{align}
		&z^{\Delta+1}\int_{s_-}^x\mf w_\Delta^*H\Bigl(\hat f-\sum_{l=0}^{\Delta-1}z^l\lambda_0\mf w_l\Bigr)
		- \sum_{l=0}^{\Delta}z^l\omega_l^-\hat f(s_-)_1
		\nonumber\\[1ex]
		&= z^{\Delta-k+1}\int_{s_-}^x \mf w_{\Delta-k}^*H
		\biggl(\hat f-\lambda_0\sum_{l=0}^{\Delta-1}z^{l+k}\mf w_{l+k}\biggr) 
		- \sum_{l=0}^{\Delta-k}z^l\omega_l^-\hat f(s_-)_1
		\nonumber\\[1ex]
		&\quad - \sum_{j=1}^k z^{\Delta-j+1}\bigl(\mf w_{\Delta-j+1}(x)\bigr)^*J
		\biggl(\hat f(x)-\lambda_0\sum_{l=0}^{\Delta-1}z^{l+j}\mf w_{l+j}(x)\biggr),
		\label{D111}
	\end{align}
	where we use $\mf w_{-1}=0$.
	For $k=0$ this is trivial.
	Assume now that \eqref{D111} holds for some $k$.
	We apply Green's identity \eqref{D108} on the interval $(s_-,x)$ and use \eqref{D64} to obtain
	\begin{align*}
		&z^{\Delta+1}\int_{s_-}^x\mf w_\Delta^* H\Bigl(\hat f-\sum_{l=0}^{\Delta-1}z^l\lambda_0\mf w_l\Bigr)
		- \sum_{l=0}^{\Delta}z^l\omega_l^-\hat f(s_-)_1
		\\[1ex]
		&= z^{\Delta-k}\int_{s_-}^x \mf w_{\Delta-k}^* H
		\biggl(z\hat f-\lambda_0\sum_{l=0}^{\Delta-1}z^{l+k+1}\mf w_{l+k}\biggr) 
		- \sum_{l=0}^{\Delta-k}z^l\omega_l^-\hat f(s_-)_1
		\\[1ex]
		&\quad - \sum_{j=1}^k z^{\Delta-j+1}\bigl(\mf w_{\Delta-j+1}(x)\bigr)^*J
		\biggl(\hat f(x)-\lambda_0\sum_{l=0}^{\Delta-1}z^{l+j}\mf w_{l+j}(x)\biggr)
		\displaybreak[0]\\[1ex]
		&= z^{\Delta-k}\biggl[\int_{s_-}^x \mf w_{\Delta-k-1}^* H
		\biggl(\hat f-\lambda_0\sum_{l=0}^{\Delta-1}z^{l+k+1}\mf w_{l+k+1}\biggr)
		\\[1ex]
		&\quad - \bigl(\mf w_{\Delta-k}(x)\bigr)^*J\biggl(\hat f(x)-\lambda_0\sum_{l=0}^{\Delta-1}z^{l+k+1}\mf w_{l+k+1}(x)\biggr)
		\\[1ex]
		&\quad + \bigl(\mf w_{\Delta-k}(s_-)\bigr)^*J\biggl(\hat f(s_-)-\lambda_0\sum_{l=0}^{\Delta-1}z^{l+k+1}\mf w_{l+k+1}(s_-)\biggr)\biggr]
		\\[1ex]
		&\quad - \sum_{l=0}^{\Delta-k}z^l\omega_l^-\hat f(s_-)_1
		- \sum_{j=1}^k z^{\Delta-j+1}\bigl(\mf w_{\Delta-j+1}(x)\bigr)^*J
		\biggl(\hat f(x)-\lambda_0\sum_{l=0}^{\Delta-1}z^{l+j}\mf w_{l+j}(x)\biggr),
	\end{align*}
	which equals the right-hand side of \eqref{D111} with $k$ replaced by $k+1$
	since $(\mf w_n(s_-))^*J\mf w_m(s_-)=0$ for all $n,m\in\bb N_0$ by \eqref{D64}.
	Hence \eqref{D111} holds for all $k\in\{0,\ldots,\Delta+1\}$.
	For $k=\Delta+1$ we obtain
	\begin{align}
		&z^{\Delta+1}\int_{s_-}^x\mf w_\Delta^*H\Bigl(\hat f-\sum_{l=0}^{\Delta-1}z^l\lambda_0\mf w_l\Bigr)
		- \sum_{l=0}^{\Delta}z^l\omega_l^-\hat f(s_-)_1
		\nonumber\\[1ex]
		&= - \sum_{j=1}^{\Delta+1} z^{\Delta-j+1}\bigl(\mf w_{\Delta-j+1}(x)\bigr)^*J
		\biggl(\hat f(x)-\lambda_0\sum_{l=0}^{\Delta-1}z^{l+j}\mf w_{l+j}(x)\biggr)
		\nonumber\\[1ex]
		&= -\sum_{n=0}^\Delta z^n\bigl(\mf w_n(x)\bigr)^*J
		\biggl(\hat f(x)-\lambda_0\sum_{l=0}^{\Delta-1}z^{\Delta-n+1+l}\mf w_{\Delta-n+1+l}(x)\biggr)
		\nonumber\\[1ex]
		&= - \sum_{n=0}^\Delta z^n\bigl(\mf w_n(x)\bigr)^*J\biggl(\hat f(x)-\lambda_0\sum_{j=\Delta-n+1}^{2\Delta-n}z^j\mf w_j(x)\biggr).
		\label{D112}
	\end{align}
	Since $J^* = -J$, we have
	\[
		\bigl(\mf w_n(x)\bigr)^*J\mf w_j(x)+\bigl(\mf w_j(x)\bigr)^*J\mf w_n(x) = 0
	\]
	for $k,l\in\bb N_0$, and hence
	\begin{align*}
		&\sum_{n=0}^\Delta\hspace*{1ex}\sum_{j=\Delta+1-n}^\Delta
		z^{j+n}\bigl(\mf w_n(x)\bigr)^*J\mf w_j(x)
		= \sum_{\substack{1\le j,n \le \Delta \\[0.5ex] \Delta+1\le j+n\le 2\Delta}}
		z^{j+n}\bigl(\mf w_n(x)\bigr)^*J\mf w_j(x)
		\\[1ex]
		&= \frac{1}{2}\hspace*{-2ex}
		\sum_{\substack{1\le j,n\le \Delta \\[0.5ex] \Delta+1\le j+n\le2\Delta}}
		z^{j+n}\Bigl(\bigl(\mf w_n(x)\bigr)^*J\mf w_j(x)
		+ \bigl(\mf w_j(x)\bigr)^*J\mf w_n(x)\Bigr)
		= 0.
	\end{align*}
	Together with \eqref{D112} we obtain \eqref{D113}.
\end{proof}

\begin{lemma}\label{D165}
	Assume that the interval $(s_-,\sigma)$ is indivisible.
	Then
	\[
		\mf w_0(t) = \binom01, \qquad \mf w_1(t) = \begin{pmatrix} -\int_{s_-}^t h_2(s)\DD s \\[0.5ex] 0 \end{pmatrix},
		\qquad \mf w_n(t) = 0 \quad\text{for} \ n\ge2,
	\]
	for $t\in[s_-,\sigma)$, and $\omega_n^-=0$ for $n\ge1$.
	Moreover, let $c=\binom{c_1}{c_2}\in\bb C^2$ and let $\hat f_-$ be the unique solution of \eqref{D2} on $[s_-,\sigma)$
	that satisfies 
	\begin{equation}\label{D181}
		\hat f_-(s_-)=c.
	\end{equation}
	Then
	\[
		\hat f_-(t) = \begin{pmatrix} c_1-zc_2\int_{s_-}^t h_2(s)\DD s \\[0.5ex] c_2 \end{pmatrix}
		= c_2\mf w_0(t) + zc_2\mf w_1(t) + c_1\binom10
	\]
	and hence $\lambda_0=c_2$ with $\lambda_0$ as in \eqref{D35} and
	\begin{equation}\label{D166}
		\Gamma^-(z)\hat f_- = c.
	\end{equation}
	The same statement holds when we consider the interval $(\sigma,s_+)$ instead of $(s_-,\sigma)$ and the corresponding
	quantities $\omega_n^+$ and $\hat f_+$.
\end{lemma}
\begin{proof}
	Most statements are easy to check.  We only prove \eqref{D166}.  
	Since $\mf w_n=0$ for $n\ge2$, we have, for $x\in(s_-,\sigma)$, (note also \cref{D119}\,(ii))
	\begin{align*}
		& \sum_{n=0}^\Delta z^n\bigl(\mf w_n(x)\bigr)^*J
		\biggl(\hat f_-(x)-(\Gammar^- \hat f_-)\sum_{j=\Delta+1}^{2\Delta-n}z^j\mf w_j(x)\biggr)
		\\[1ex]
		&= \sum_{n=0}^1 z^n\bigl(\mf w_n(x)\bigr)^*J\hat f_-(x)
		\\[1ex]
		&= \begin{pmatrix} -z\int_{s_-}^x h_2(s)\DD s \\[1ex] 1 \end{pmatrix}^*
		J\begin{pmatrix} c_1-zc_2\int_{s_-}^t h_2(s)\DD s \\[0.5ex] c_2 \end{pmatrix}
		= c_1,
	\end{align*}
	which, together with \eqref{D135} proves \eqref{D166}.
\end{proof}

\begin{Remark}\label{D174}
	For some considerations it is useful to replace either $H_-$ or $H_+$ with a Hamiltonian
	so that the corresponding interval becomes indivisible.
	Assume that $(s_-,\sigma)$ is not indivisible.
	Let $\mr H_+(t)\DE\bigl(\begin{smallmatrix} 0 & 0 \\ 0 & h_2(t) \end{smallmatrix}\bigr)$,
	$t\in(\sigma,s_+)$, where $h_2$ is not integrable at $\sigma$, set $H^{-,0}\DE(H_-,\mr H_+)$ 
	and define the elementary indefinite Hamiltonian of kind (A) $\mf h^{-,0}\DE(H^{-,0};0,0;0)$.
	Note that $\Delta(H^{-,0})=\Delta(H_-)$ and $\ms p_{-,0}=0$ with $\ms p_{-,0}$ defined as in \eqref{D176}.
	If $\Delta=\Delta(H_-)\ge\Delta(H_+)$, then the generalised boundary mappings $\Gammas^-(z)$
	for $\mf h$ and $\mf h^{-,0}$ coincide.

	In a similar way one can define an elementary indefinite Hamiltonian of kind (A) $\mf h^{+,0}$ by replacing $H_-$
	with a Hamiltonian of the form $\bigl(\begin{smallmatrix} 0 & 0 \\ 0 & h_2(t) \end{smallmatrix}\bigr)$ 
	under the assumption that $(\sigma,s_+)$ is not indivisible.
\end{Remark}

\begin{proof}[Proof of \cref{D115}]
	If $(s_-,\sigma)$ is not indivisible, then $f$ is uniquely determined on $(s_-,\sigma)$ 
	and satisfies \eqref{D130} by \cref{D72} and \eqref{D60}.
	If $(s_-,\sigma)$ is indivisible, then $f_2$ is uniquely determined and, with the notation from \cref{D165} 
	and by \eqref{D60}, satisfies
	\begin{equation}\label{D175}
		\hat f(s_-)_2 = \hat f(x)_2 = c_2 = \lambda_0 = \mrr\Gamma(F;zF)(s_-)_2
	\end{equation}
	for $x\in[s_-,\sigma)$;
	moreover, $f_1$ is unique up to an additive constant by \cref{D165}, which can be chosen so that \eqref{D130} holds.

	It follows from \cite[Theorem~4.21\,(iv)]{langer.woracek:gpinf} that $\lim_{x\to\sigma}\hat f(x)_2$ exists,
	which proves \eqref{D136}.
	It follows from \eqref{D175} that $\Gammar^-\hat f_-=\lambda_0$ if $(s_-,\sigma)$ is indivisible
	and that $\Gammar^+\hat f_+=\lambda_0$ if $(\sigma,s_+)$ is indivisible.
	If neither $(s_-,\sigma)$ nor $(\sigma,s_+)$ is indivisible, then
	we use the Hamiltonian $\mf h^{-,0}$ from \cref{D174} to obtain $\Gammar^-\hat f_-=\Gammar^+\hat f_+=\lambda_0$ 
	by what we have already proved since $\Gammar^\pm$ is independent of $H$.
	
	Finally, we prove \eqref{D137}. It follows from \cref{D121,D122} that
	\begin{align*}
		& \lambda_0\ms p(z) + \mrr\Gamma(F;zF)(s_-)_1 - \mrr\Gamma(F;zF)(s_+)_1
		\\[1ex]
		&= \lim_{x\to\sigma^-} z^{\Delta+1}\int_{s_-}^x \mf w_\Delta^*H\Bigl(f-\sum_{l=0}^{\Delta-1}z^l\lambda_0\mf w_l\Bigr)
		- \sum_{l=0}^{\Delta}z^l\omega_l^-\hat f(s_-)_1 + \hat f_1(s_-)_1
		\\
		&\quad + \lim_{x\to\sigma^+} z^{\Delta+1}\int_x^{s_+} \mf w_\Delta^*H\Bigl(f-\sum_{l=0}^{\Delta-1}z^l\lambda_0\mf w_l\Bigr)
		+ \sum_{l=0}^{\Delta}z^l\omega_l^+\hat f(s_+)_1 - \hat f(s_+)_1
		\\[1ex]
		&= -\lim_{x\to\sigma^-}\sum_{n=0}^\Delta z^n\bigl(\mf w_n(x)\bigr)^*J
		\biggl(\hat f(x)-\lambda_0\sum_{j=\Delta+1}^{2\Delta-n}z^j\mf w_j(x)\biggr) + \hat f_1(s_-)_1
		\\
		&\quad + \lim_{x\to\sigma^+}\sum_{n=0}^\Delta z^n\bigl(\mf w_n(x)\bigr)^*J
		\biggl(\hat f(x)-\lambda_0\sum_{j=\Delta+1}^{2\Delta-n}z^j\mf w_j(x)\biggr) - \hat f_1(s_+)_1,
	\end{align*}
	which, together with $\lambda_0=\Gammar^-\hat f_-$, proves \eqref{D137}.
\end{proof}

\medskip

\noindent
The next theorem shows that the generalised boundary mappings $\Gamma^\pm(z)$ are bijective from 
the set of locally absolutely continuous solutions of \eqref{D2} on $[s_-,\sigma)$ and $(\sigma,s_+]$ respectively
onto $\bb C^2$.

\begin{theorem}\label{D164}
Let $c\in\bb C^2$ and $z\in\bb C$.
\begin{Enumerate}
\item
	There exists a unique locally absolutely continuous solution $\hat f_-$ of \eqref{D2} on $(s_-,\sigma)$ such that $\Gamma^-(z)\hat f_-=c$.
\item
	There exists a unique locally absolutely continuous solution $\hat f_+$ of \eqref{D2} on $(\sigma,s_+)$ such that $\Gamma^+(z)\hat f_+=c$.
\end{Enumerate}
\end{theorem}

\begin{proof}
	During this proof we use the notation $\Gamma(\mf h)^\pm(z)$ instead of $\Gamma^\pm(z)$ to indicate the 
	dependence on the elementary indefinite Hamiltonian of kind (A) $\mf h$.
	
	Let us first prove existence in (i) in the case when $\Delta=\Delta(H_-)\ge\Delta(H_+)$.
	We consider the elementary indefinite Hamiltonian of kind (A) $\mf h^{-,0}$ from \cref{D174}.
	By \cref{D39}\,(ii) there exists a unique $F\in\ker\bigl(\mrr T(\mf h^{-,0})-z\bigr)$ 
	such that $\mrr\Gamma(\mf h^{-,0})(F;zF)(s_+)=c$.
	Further, choose the solution $\hat f_+$ such that $\hat f_+(s_+)=c$, which can be done by \cref{D115}\,(ii).
	Since $\Delta(H^{-,0})=\Delta$ and the Hamiltonians on $(s_-,\sigma)$ are the same for $\mf h$ and $\mf h^{-,0}$, 
	we have $\Gamma(\mf h)^-(z)=\Gamma(\mf h^{-,0})^-(z)$.
	Hence, we obtain from \eqref{D177}, \eqref{D181} and \eqref{D166} that
	\[
		\Gamma(\mf h)^-(z)\hat f_- = \Gamma(\mf h^{-,0})^-(z)\hat f_-
		= \Gamma(\mf h^{-,0})^+(z)\hat f_+ = \hat f^+(s_+) = c.
	\]
	The proof of the existence in (ii) when $\Delta=\Delta(H_+)$ is similar.
	
	Let us now prove existence in (i) when $\Delta=\Delta(H_+)>\Delta(H_-)$.  Set $d\DE\ms R(z)c$.
	It follows from what we have already proved that there exists a solution $\hat f_+$ of \eqref{D2} on $(\sigma,s_+)$
	such that $\Gamma^+(z)\hat f_+=d$.
	By \cref{D39}\,(ii) there exists $F=(f,\upxi,\upalpha)\in\ker(\mrr T(\mf h)-z)$ with $\mrr\Gamma(\mf h)(F;zF)(s_+)=\hat f_+(s_+)$.
	Since $\Delta(H_+)>1$, the interval $(\sigma,s_+)$ is not indivisible and hence $\hat f_+$ 
	is the unique locally absolutely continuous representative of $f_+$.
	Let $\hat f_-$ be the locally absolutely continuous representative of $f_-$ 
	that satisfies $\hat f_-(s_-)=\mrr\Gamma(\mf h)(F;zF)(s_-)$.
	It follows from \eqref{D177} that
	\[
		\Gamma^-(z)\hat f_- = (\ms R(z))^{-1}\Gamma^+(z)\hat f_+
		= (\ms R(z))^{-1}d = c.
	\]
	The proof of (ii) in the case when $\Delta>\Delta(H_+)$ is analogous.
	
	Finally, uniqueness follows from the fact that the space of solutions of \eqref{D2} 
	on $(s_-,\sigma)$ (or $(\sigma,s_+)$ respectively) is two-dimensional.
%
\end{proof}

\subsection{Factorisation of the monodromy matrix}
\label{D160}

Let us denote the entries of $W_{\mf h}(x,z)$ by $w_{ij}(x,z)$, $i,j=1,2$.
The transposes of the rows of $W_{\mf h}(\Dummy,z)$, i.e.\ $\binom{w_{i1}(\Dummy,z)}{w_{i2}(\Dummy,z)}$, $i=1,2$,
are solutions of \eqref{D2} on $(s_-,\sigma)$ and $(\sigma,s_+)$.
By this property and $W_{\mf h}(s_-,z)=I$ the matrix $W_{\mf h}(t,z)$ is determined on the interval $[s_-,\sigma)$, but not on the
interval $(\sigma,s_+]$. 

In the next theorem, the main result of the paper, we construct $W_{\mf h}$ 
to the right of the singularity using an interface condition relating regularised boundary values. 
Let us first introduce some notation.  
Set $\mf e_1=\binom10$, $\mf e_2=\binom01$. Given a matrix $M(t,z)$ whose columns are solutions of \eqref{D2} on the interval
$[s_-,\sigma)$ or $(\sigma,s_+]$ respectively, define 
\[
	\tilde\Gamma^\pm(z)M(\Dummy,z) \DE
	\Big( \Gamma^\pm(z)\big(M(\Dummy,z)\mf e_1\big)\quad \Gamma^\pm(z)\big(M(\Dummy,z)\mf e_2\big) \Big)\in\bb C^{2\times 2}
\]
for $z\in\bb C$.

\begin{theorem}\label{D123}
	Let $V(\Dummy,z)$ be a solution of \eqref{D70} on $(\sigma,s_+]$ such that $V(t,z)$ is non-singular
	for all $t\in(\sigma,s_+]$ and $z\in\bb C$.
	Further, set
	\begin{equation}\label{D178}
		U^-(z) \DE \bigl[\tilde\Gamma^-(z)W_{\mf h}(\Dummy,z)^T\bigr]^T, \qquad
		U^+(z) \DE \bigl[\tilde\Gamma^+(z)V(\Dummy,z)^T\bigr]^T.
	\end{equation}
	Then
	\begin{equation}\label{D127}
		\det U^-(z) = 1, \qquad \det U^+(z) = \det V(t,z)
	\end{equation}
	and
	\begin{equation}\label{D124}
		W_{\mf h}(t,z) = U^-(z)\bigl(\ms R(z)\bigr)^T\bigl(U^+(z)\bigr)^{-1}V(t,z)
	\end{equation}
	for $z\in\bb C$ and $t\in(\sigma,s_+]$.
\end{theorem}

\begin{remark}\label{D171}
\rule{0ex}{1ex}
\begin{Enumerate}
\item
	The rows of $U^-(z)$ contain the transposes of the boundary values $\Gamma^-(z)\binom{w_{i1}(\Dummy,z)}{w_{i2}(\Dummy,z)}$
	of the transposes of the rows of $W_{\mf h}$.
\item
	The transposes of the rows of $(U^+(z))^{-1}V(\Dummy,z)$ are solutions of \eqref{D2} and their
	generalised boundary values $\Gamma^+(z)\Dummy$ are $\binom10$ and $\binom01$ respectively.
\item
	The matrices $U^-(z)$ and $(U^+(z))^{-1}V(t,z)$ depend only on $H$ whereas $\ms R(z)$ depends
	only on the discrete data $d_j$, $\oe$ and $b_j$.
\item
	The factorisation formula \eqref{D124} extends also to the case when $\mf h$ is an elementary indefinite 
	Hamiltonian of kind (B) or (C).  In this case both $(s_-,\sigma)$ and $(\sigma,s_+)$ are indivisible
	and hence $U^-(z)=I$ and $(U^+(z))^{-1}V(s_+,z)=I$ by \cref{D165}.
	In particular, $W_{\mf h}(s_+,z)=(\ms R(z))^T$, which coincides with the result 
	in \cite[Proposition~4.31]{kaltenbaeck.woracek:p5db} if one replaces $d_1$ by $-b_{\oe+1}$.
\end{Enumerate}
\end{remark}

\begin{proof}[Proof of \cref{D123}]
	We first show \eqref{D124}.
	Denote the right-hand side of \eqref{D124} by $\tildeW(t,z)$ for $t\in(\sigma,s_+]$
	and $z\in\bb C$.
	Let $i\in\{1,2\}$ and $z\in\bb C$ and let $F^{(i)}=(f^{(i)},\upxi^{(i)},\upalpha^{(i)})$
	be the unique element that satisfies 
	\[
		F^{(i)}\in\ker\bigl(\mrr T(\mf h)-z\bigr)
		\qquad\text{and}\qquad
		\mrr\Gamma(\mf h)(F^{(i)};zF^{(i)})(s_-)=\mf e_i;
	\]
	see \cref{D39}\,(ii).
	Further, let $\hat f\pm^{(i)}$ be the unique locally absolutely continuous representatives of $f_\pm^{(i)}$
	such that \eqref{D130} and \eqref{D131} hold.
	The transpose of the $i$th row of $W_{\mf h}$, $\hat y_-^{(i)}\DE W_{\mf h}(\Dummy,z)^T\mf e_i$,
	is a solution of \eqref{D2} on $(s_-,\sigma)$ and satisfies $\hat y_-^{(i)}(s_-)=\mf e_i$,
	and hence $\hat f_-^{(i)}=\hat y_-^{(i)}$.

	The transpose of the $i$th row of $\tildeW$, $\hat y^{(i)}_+=\tildeW(\Dummy,z)^T\mf e_i$,
	satisfies \eqref{D2} on $(\sigma,s_+)$, and for the generalised boundary values 
	we obtain from \eqref{D178} and \eqref{D177} that
	\begin{align*}
		\Gamma^+(z)\hat y_+^{(i)} 
		&= \tilde\Gamma^+(z)V(\Dummy,z)^T\bigl[\bigl(U^+(z)\bigr)^{-1}\bigr]^T
		\ms R(z)\bigl(U^-(z)\bigr)^T \mf e_i
		\\[1ex]
		&= \ms R(z)\tilde\Gamma^-(z)W_{\mf h}(\Dummy,z)^T \mf e_i
		= \ms R(z)\Gamma^-(z)\hat f_-^{(i)}
		= \Gamma^+(z)\hat f_+^{(i)}.
	\end{align*}
	The uniqueness statement of \cref{D164}\,(ii) implies that $\hat f_+^{(i)}=\hat y_+^{(i)}$
	and hence
	\[
		\mrr\Gamma(\mf h)(F^{(i)};zF^{(i)})(s_+) = \hat y_+^{(i)}(s_+) = \tildeW(s_+,z)^T\mf e_i.
	\]
	Together with \cite[Theorem~5.1 and Definition~4.3]{kaltenbaeck.woracek:p5db},
	this shows that $W_{\mf h}(s_+,z)=\tildeW(s_+,z)$, and hence \eqref{D124} holds for $t\in(\sigma,s_+]$ 
	since $\tildeW$ satisfies \eqref{D70}.
	
	Let us now show \eqref{D127}.
	Taking determinants on both sides of \eqref{D124} we obtain
	\begin{equation}\label{D179}
		1 = \det\bigl(U^-(z)\bigr)\frac{\det(V(t,z))}{\det(U^+(z))}.
	\end{equation}
	Assume first that $\Delta(H_-)=\Delta$.
	Consider the elementary indefinite Hamiltonian of kind (A) $\mf h^{-,0}$ from \cref{D174}
	and choose
	\[
		\mr V(t,z) = \begin{pmatrix} 1 & 0 \\[1ex] z\int_t^{s_+}h_2(s)\DD s & 1 \end{pmatrix}.
	\]
	Then $\mr U^+(z)=I$ and $\det(\mr V(t,z))=1$, and hence \eqref{D179} yields $\det(U^-(z))=1$.
	Now we can use \eqref{D179} with the given elementary indefinite Hamiltonian of kind (A) $\mf h$, which implies 
	the second relation in \eqref{D127}.
	
	Finally, let us consider the case when $\Delta(H_-)<\Delta$.
	We use the elementary indefinite Hamiltonian of kind (A) $\mf h^{+,0}$ from \cref{D174}, for which we have $\mr U^-(z)=I$,
	and hence \eqref{D179} implies $\det(U^+(z))=\det(V(t,z))$.
	With the given $\mf h$ we then obtain from \eqref{D179} that $\det(U^-(z))=1$.
\end{proof}

\medskip

\noindent
In the following corollary we compare the monodromy matrices for two different sets of discrete parameters
while keeping $H$ fixed.

\begin{corollary}\label{D180}
	Let two elementary indefinite Hamiltonians of kind (A), $\mf h_1$ and $\mf h_2$, be given whose
	Hamiltonians coincide, $H_1=H_2$, and with polynomials $\ms p_1$ and $\ms p_2$ respectively,
	and set
	\begin{equation}\label{D182}
		\ms M(z) \DE \lim_{x\to\sigma}\begin{pmatrix} w_{\mf h_1,12}(x,z) \\[1ex] w_{\mf h_1,22}(x,z) \end{pmatrix}
		\bigl(w_{\mf h_1,22}(x,z) \;\; -w_{\mf h_1,12}(x,z)\bigr).
	\end{equation}
	Then
	\begin{equation}\label{D183}
		W_{\mf h_2}(t,z)-W_{\mf h_1}(t,z) = \bigl(\ms p_2(z)-\ms p_1(z)\bigr)\ms M(z)W_{\mf h_1}(t,z)
	\end{equation}
	for $t\in(\sigma,s_+]$ and $z\in\bb C$.
\end{corollary}

\pagebreak[3]

\begin{Remark}\label{D184}
\rule{0ex}{1ex}
\begin{Enumerate}
\item
	With the intermediate Weyl coefficient
	\[
		q_\sigma(z) \DE \lim_{x\to\sigma^-}\frac{w_{\mf h_1,12}(x,z)}{w_{\mf h_1,22}(x,z)},
		\qquad z\in\bb C\setminus\bb R,
	\]
	(which is the Weyl coefficient of the Hamiltonian $H_-$) one can rewrite $\ms M(z)$ as
	\[
		\ms M(z) = \Bigl(\lim_{x\to\sigma^-}w_{\mf h_1,22}(x,z)\Bigr)^2\begin{pmatrix} q_\sigma(z) \\[1ex] 1 \end{pmatrix}
		\bigl(1 \;\; -q_\sigma(z)\bigr)
	\]
	for $z\in\bb C\setminus\bb R$.
\item
	\Cref{D180} is related to \cite[Theorem~5.4]{langer.woracek:lemneu} and improves it.
	In that paper only the case of one negative square is treated.  
	Moreover, \eqref{D183} simplifies the result in \cite{langer.woracek:lemneu} substantially
	as only limits of entries of $W_{\mf h_1}$ are needed and no derivatives with respect to the spectral parameter.
	Note that \cite[Theorem~5.4]{langer.woracek:lemneu} actually describes the change of the Weyl coefficient
	when $H_-$ is in limit point case at $s_-$ rather than the change of the monodromy matrix.
\end{Enumerate}
\end{Remark}

\begin{proof}[Proof of \cref{D180}]
	Choose $V=W_{\mf h_1}$ in \cref{D123} and write 
	$U^\pm(z)=(u_{ij}^\pm(z))_{i,j=1}^2$ 
	It follows from \eqref{D124} that
	\begin{align*}
		W_{\mf h_2}(t,z)-W_{\mf h_1}(t,z) 
		&= U^-(z)\begin{pmatrix} 0 & 0 \\[0.5ex] \ms p_2(z)-\ms p_1(z) & 0 \end{pmatrix}\bigl(U^+(z)\bigr)^{-1}W_{\mf h_1}(t,z)
		\\[1ex]
		&= \bigl(\ms p_2(z)-\ms p_1(z)\bigr)U^-(z)\begin{pmatrix} 0 \\ 1 \end{pmatrix}
		\bigl(1 \;\;\; 0 \bigr)\bigl(U^+(z)\bigr)^{-1}W_{\mf h_1}(t,z)
	\end{align*}
	for $t\in(\sigma,s_+]$ and $z\in\bb C$.
	Since $\det U^+(z)=\det W_{\mf h_1}(t,z)=1$ by \eqref{D127}, we obtain
	\begin{align*}
		& U^-(z)\begin{pmatrix} 0 \\ 1 \end{pmatrix}
		\bigl(1 \;\;\; 0 \bigr)\bigl(U^+(z)\bigr)^{-1}
		= \begin{pmatrix} u_{12}^-(z) \\[1ex] u_{22}^-(z) \end{pmatrix}
		\bigl(u_{22}^+(z) \;\; -u_{12}^+(z)\bigr)
		\\[1ex]
		&= \begin{pmatrix} \lim\limits_{x\to\sigma^-}w_{\mf h_1,12}(x,z) \\[2ex] \lim\limits_{x\to\sigma^-}w_{\mf h_1,22}(x,z) \end{pmatrix}
		\Bigl(\lim\limits_{x\to\sigma^+}w_{\mf h_1,22}(x,z) \;\; \lim\limits_{x\to\sigma^+}w_{\mf h_1,12}(x,z)\Bigr),
	\end{align*}
	which proves \eqref{D183} with \eqref{D182}.
\end{proof}

\subsection{An example}
\label{D162}

In this section we revisit an example that is studied in \cite{langer.langer.sasvari:2004}.
Let $s_-=0$, $s_+>1$ and consider the Hamiltonian
\[
	H(t) = \begin{pmatrix} (t-1)^2 & 0 \\[1ex] 0 & \frac{1}{(t-1)^2} \end{pmatrix},
	\qquad t\in(0,s_+)\setminus\{1\}.
\]
It is easy to see that $H$ satisfies \textsf{(I)} and \textsf{(HS)}.
Since $H$ is diagonal, we can use \cref{D63}\,(ii) to find $\mf w_1$, namely,
\begin{equation}\label{D167}
	\mf w_1(t) 
	= \begin{pmatrix} -\int_{s_\pm}^t \frac{1}{(s-1)^2}\DD s \\[1ex] 0 \end{pmatrix}
	= \begin{pmatrix} \frac{1}{t-1}-\frac{1}{s_\pm-1} \\[1ex] 0 \end{pmatrix},
	\qquad t\in I_\pm,
\end{equation}
where $I_-\DE(0,1)$, $I_+\DE(1,s_+)$.
Clearly, $\mf w_1(t)\in L^2(H)$, which implies that $\Delta(H)=1$.

The generalised boundary mappings $\Gammas^\pm(z)$ from \eqref{D135},
for which we only need the functions $\mf w_0$ and $\mf w_1$ by \cref{D119}\,(ii),
are given by
\begin{align*}
	\Gammas^\pm(z)f &= \lim_{x\to1^\pm}\Bigl(\mf w_0(x)^*Jf(x)+z\mf w_1(x)^*Jf(x)\Bigr)
	\\[1ex]
	&= \lim_{x\to1^\pm}\Bigl[f(x)_1-z\Bigl(\frac{1}{x-1}-\frac{1}{s_\pm-1}\Bigr)f(x)_2\Bigr].
\end{align*}

Let us now consider an elementary indefinite Hamiltonian of kind (A) $\mf h=\langle H;\oe,b_j;d_j\rangle$ with $d_0,d_1\in\bb R$,
$\oe\in\bb N_0$, and real numbers $b_1,\ldots,b_\oe$ with $b_1\ne0$ if $\oe>0$.
It is easy to see (cf.\ \cite[Theorem~7.1]{langer.langer.sasvari:2004}) that
the matrix function
\begin{equation}\label{D192}
	W(t,z) =
	\begin{pmatrix}
		\dfrac{\sin(zx)-z\cos(zx)}{z(x-1)} & \Bigl(\dfrac{1}{z^2}-(x-1)\Bigr)\sin(zx)-\dfrac{x\cos(zx)}{z}
		\\[2.5ex]
		\dfrac{\sin(zx)}{x-1} & \dfrac{\sin(zx)}{z}-(x-1)\cos(zx)
	\end{pmatrix}
\end{equation}
satisfies \eqref{D70} on $(0,s_+)\setminus\{1\}$ and the relation $W(0,z)=I$.
Hence $W_{\mf h}(t,z)=W(t,z)$ for $t\in[0,1)$ and $z\in\bb C$.
Moreover, we can use $W$ as the matrix $V$ in \cref{D123}.
Let $U^\pm(z)$ be defined as in \eqref{D178}.
An elementary calculation shows that
\begin{align*}
	U^-(z)
	&=
	\begin{pmatrix}
		z\sin z-\dfrac{\sin z}{z}+2\cos z \; & \dfrac{\sin z}{z^2}-\dfrac{\cos z}{z}
		\\[2ex]
		z\cos z-\sin z & \dfrac{\sin z}{z}
	\end{pmatrix},
	\\[1ex]
	\bigl(U^+(z)\bigr)^{-1}
	&=
	\begin{pmatrix}
		\dfrac{\sin z}{z} & -\dfrac{\sin z}{z^2}+\dfrac{\cos z}{z}
		\\[2ex]
		-z\cos z-\dfrac{\sin z}{s_+-1} & \; z\sin z+\dfrac{\sin z}{z(s_+-1)}+\cos z-\dfrac{\cos z}{s_+-1}
	\end{pmatrix}.
\end{align*}
With $\ms R(z)$ defined as in \eqref{D189} we obtain from \cref{D123} and with another lengthy 
but elementary calculation that, for $t\in(1,s_+]$ and $z\in\bb C$,
\begin{align*}
	W_{\mf h}(t,z) &= U^-(z)\bigl(\ms R(z)\bigr)^T\bigl(U^+(z)\bigr)^{-1}W(t,z)
	\\[1ex]
	&= W(t,z) + \Bigl(\ms p(z)-\frac{s_+}{s_+-1}z\Bigr)\ms N(z)W(t,z)
\end{align*}
where
\[
	\ms N(z) =
	\begin{pmatrix}
		\dfrac{\sin z}{z}\Bigl(\dfrac{\sin z}{z^2}-\dfrac{\cos z}{z}\Bigr) & -\Bigl(\dfrac{\sin z}{z^2}-\dfrac{\cos z}{z}\Bigr)^2
		\\[2ex]
		\dfrac{\sin^2 z}{z^2} & -\dfrac{\sin z}{z}\Bigl(\dfrac{\sin z}{z^2}-\dfrac{\cos z}{z}\Bigr)
	\end{pmatrix}.
\]
In particular, if we choose
\begin{equation}\label{D190}
	d_0=-\frac{s_+}{s_+-1}, \quad d_1=0, \quad \oe=0,
\end{equation}
then $W_{\mf h}=W$.
It is easy to see that the matrix $\ms N(z)$ equals $\ms M(z)$ in \cref{D180} 
if $\mf h_1$ is chosen with the parameters in \eqref{D190}.



{\footnotesize
\begin{flushleft}
	M.~Langer \\
	Department of Mathematics and Statistics \\
	University of Strathclyde \\
	26 Richmond Street \\
	Glasgow G1 1XH \\
	UNITED KINGDOM \\
	email: \texttt{m.langer@strath.ac.uk} \\[5mm]
\end{flushleft}
\begin{flushleft}
	H.\,Woracek\\
	Institute for Analysis and Scientific Computing\\
	Vienna University of Technology\\
	Wiedner Hauptstra{\ss}e\ 8--10/101\\
	1040 Wien\\
	AUSTRIA\\
	email: \texttt{harald.woracek@tuwien.ac.at}\\[5mm]
\end{flushleft}
}

\end{document}